\documentclass[11pt]{amsart}
\usepackage{graphicx} % Required for inserting images
\usepackage{geometry} % For adjusting margins
\usepackage{hyperref} % Makes clickable references
\usepackage{wasysym}  % fuer \varhexstar 
\usepackage[all]{xy}
\usepackage{pifont} % fuer \checkmark
\usepackage{psfrag}
\usepackage{todonotes}
\usepackage{amsmath,amssymb,amsthm} % Maths symbols
\usepackage[capitalize]{cleveref} % Clever referencing
\usepackage{enumitem} % Enumerating with various bullets
\usepackage{tikz} % For drawing diagrams
\usetikzlibrary{arrows,shapes,trees,cd,calc,decorations.pathmorphing}
\usepackage{comment} 
\usepackage{xcolor}  % for coloured comments, added by Karin
\tikzset{curve/.style={settings={#1},to path={(\tikztostart)
    .. controls ($(\tikztostart)!\pv{pos}!(\tikztotarget)!\pv{height}!270:(\tikztotarget)$)
    and ($(\tikztostart)!1-\pv{pos}!(\tikztotarget)!\pv{height}!270:(\tikztotarget)$)
    .. (\tikztotarget)\tikztonodes}},
    settings/.code={\tikzset{quiver/.cd,#1}
        \def\pv##1{\pgfkeysvalueof{/tikz/quiver/##1}}},
    quiver/.cd,pos/.initial=0.35,height/.initial=0}

\tikzset{tail reversed/.code={\pgfsetarrowsstart{tikzcd to}}}
\tikzset{2tail/.code={\pgfsetarrowsstart{Implies[reversed]}}}
\tikzset{2tail reversed/.code={\pgfsetarrowsstart{Implies}}}
\tikzset{no body/.style={/tikz/dash pattern=on 0 off 1mm}}

\title{Flip graphs of coloured triangulations of convex polygons}
\author{Karin Baur} 
\email{ka.baur@me.com}
\author{Diana Bergerova} 
\email{D.Bergerova@sms.ed.ac.uk}
\author{Jenni Voon}
\email{jyyv2@cam.ac.uk}
\author{Lejie Xu}
\email{L.Xu-43@sms.ed.ac.uk}
\date{\today}

\theoremstyle{plain}% default
\newtheorem{thm}{Theorem}[section]
\newtheorem{lem}[thm]{Lemma}
\newtheorem{prop}[thm]{Proposition}

\newtheorem{coroll}[thm]{Corollary}
\newtheorem{conj}{Conjecture}

\theoremstyle{definition}
\newtheorem{defn}[thm]{Definition}

\newtheorem{exmp}[thm]{Example}
\newtheorem{ob}{Observation}

\theoremstyle{remark}

\newtheorem{remark}[thm]{Remark}
\newtheorem*{note}{Notation}

\begin{document}

\begin{abstract}
A triangulation of a polygon is a subdivision of it into triangles, using diagonals between its vertices. Two different triangulations of a polygon can be related by a sequence of flips: a flip replaces a diagonal by the unique other diagonal in the quadrilateral it defines. 
In this paper, we study coloured triangulations and coloured flips. 
In this more general situation, 
it is no longer true that any two triangulations can be linked by a sequence of (coloured) flips. 
In this paper, we study the connected components of the coloured flip graphs of triangulations. The motivation for this is a result of  
Gravier and Payan proving that the Four-Colour Theorem is equivalent to the connectedness of the flip graph of 2-coloured triangulations. 
\end{abstract}

\maketitle

\tableofcontents

%%%%%%%%%%%%%%%
%
\section{Introduction}\label{sec:intro}

A triangulation of a polygon is a subdivision of it into triangles, using diagonals between its vertices. Two different triangulations of a polygon can be related by a sequence of flips: a flip replaces a diagonal by the unique other diagonal in the quadrilateral it defines. 
In this paper, we study $n$-coloured triangulations and $n$-coloured flips: we allocate $n$ colours to the triangles and flip diagonals only if the two triangles incident with it have the same colour, say $i$. The flip then changes to colour of the two triangles according to the colour $i+1$ (reducing modulo $n$). 
When using colours, 
it is no longer true that any two triangulations can be linked by a sequence of (coloured) flips. 
In this paper, we study the connected components of the coloured flip graphs of triangulations. The motivation for this is a result of  
Gravier and Payan proving that the Four-Colour Theorem is equivalent to the connectedness of the flip graph of 2-coloured triangulations.

This article is structured as follows:  Section~\ref{sec:background} contains the background on triangulated polygons and introduces coloured triangulations. Then it explains the link between coloured triangulations and the Four-Colour theorem. In Section~\ref{sec:components}, we study the size and structure of the connected components of the coloured flip graph. Section~\ref{sec:observations} contains further observations and a conjecture.

%%%%%%%%%%%
%
\section{Background}\label{sec:background}

Here we recall the notions of triangulations of convex polygons. 
We write $P_n$ to denote a convex polygon with $n$ vertices. 

\begin{defn}[Triangulation]
    A \emph{triangulation} of $P_n$ is a subdivision of the polygon into triangles, using pairwise non-crossing diagonals. 
\end{defn} 

Boundary segments are not considered to be diagonals. 
Note that any triangulation of $P_n$ decomposes the polygon into $n-2$ triangles, using $n-3$ diagonals.

\begin{exmp} 
    A triangulation given by $n-3$ diagonals incident with a common vertex will be called a {\em fan triangulation}. An example of a fan triangulation of a hexagon 
    is in \cref{fig:fantriangulationexmp}. 

\begin{figure}[h!]
    \centering
    \begin{tikzpicture}[scale=.5]
      \tikzstyle{invisible} = [circle, fill, outer sep=1,inner sep=1,minimum size=0]
        \draw (-0.5,4.5) node[invisible] (v1) {} -- (1.6,3.2) node[invisible] {} -- (1.6,0.4) node[invisible] (v4) {} -- (-0.5,-1) node[invisible] (v3) {} -- (-2.6,0.4) node[invisible] (v2) {} -- (-2.6,3.2) node[invisible] {} -- (v1);
        \draw  (v1) edge (v2);
        \draw  (v1) edge (v3);
        \draw  (v1) edge (v4);
 %       \draw (6.2,4.6) node[invisible] (v5) {} -- (4,3.2) node[invisible] (v8) {} -- (4,0.2) node[invisible] (v6) {} -- (6.2,-1.2) node[invisible] (v10) {} -- (8.5,0.2) node[invisible] (v9) {} -- (8.5,3.2) node[invisible] (v7) {} -- (v5);
%        \draw  (v6) edge (v7);
%        \draw  (v8) edge (v9);
%        \draw  (v10) edge (v5);
    \end{tikzpicture}
    \caption{A fan triangulation of a hexagon.}
    \label{fig:fantriangulationexmp}
\end{figure}
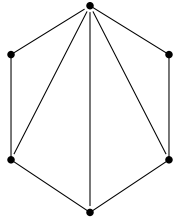
\end{exmp}

The following result is well-known. We include a proof for convenience. The strategy of the proof is illustrated for $n=8$ in \cref{fig:triangulationscountingoctagon}. 

\begin{lem}
\label{lem:catalannumbersgivenumberoftriangulations}
    The number of triangulations of a convex $(n+2)$-gon is given by the $n$-th Catalan number $C_{n} = \frac{1}{n+1}\binom{2n}{n}$.
\end{lem}

\begin{proof}
    The proof can be done using an inductive argument. One checks that the claim is true for $n=1$. Choose an edge $E$, and consider the triangle it is a part of. In an $(n+2)$-gon, there are $n$ other options for the third vertex of this triangle. All of these reduce the problem to one or two smaller cases, as to the left and right of this triangle, there are smaller polygons of size $m-1$ and $n+4-m$ respectively, for $m=3,\dots, n+2$. 
    (For $m=3$, there is only a polygon of size $n+1$ on the right of the triangle, for $m=n+2$, there is only a polygon of size $n+1$ on the left of the triangle.) We count the number of triangulations of these two subpolygons and let $m$ run: 
    This gives the total number of triangulations as $C_{n-1} + C_1 C_{n-2} + ... + C_{n-2}C_{1} + C_{n-1}$, which is a well-known recursive formula for the Catalan numbers.
\end{proof}

%\begin{exmp}
%\cref{fig:triangulationscountingoctagon} depicts a demonstration of the proof above for the octagon.
    \begin{figure}[h!]
    \includegraphics{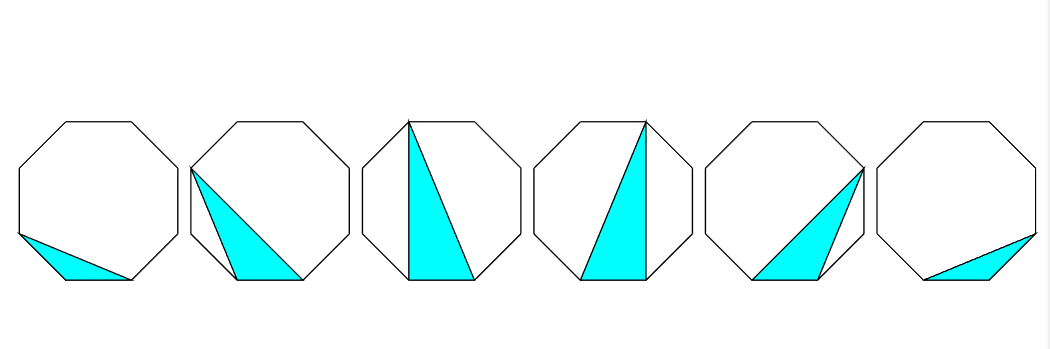}
    \caption{Each possible triangulation of the octagon falls into one of these 6 types.}
    \label{fig:triangulationscountingoctagon}
    \end{figure}
%\end{exmp}

There is a well known move on triangulated surfaces: 
\begin{defn}\label{def:flip} %[\cite{williams2013cluster}, Definition 2.15]
    Let $t$ be a diagonal in the triangulation $T$ of $P_n$. This defines a quadrilateral of the two triangles containing $t$. Then there is a new triangulation $T'$ which is obtained by replacing the diagonal $t$ with the other diagonal of that quadrilateral. This local move is called a \emph{flip}. 
\end{defn}

It is a classical result that any two triangulation of a polygon can be linked by a sequence of flips (\cite{hatcher}). 

\begin{defn}\label{def:flip-graph}
    The \emph{flip graph} of $P_n$ is the graph whose vertices are triangulations of the polygon, and two vertices $T_1, T_2$ are connected by an edge if and only if there exists a (single) flip linking $T_1$ with $T_2$. 
\end{defn}

\subsection{Coloured triangulations, coloured flips}

In this article, we are interested in a generalisation of triangulations: we equip triangulations with a set of colours and define a new flip operation for them. 

Let $m\ge 1$ and let 
$C=\{1,\dots, m\}$ be a set of $m$ different colours. If $T$ is a triangulation of a polygon, we write $F(T)$ for the set of its triangles (faces). 

\begin{defn}[Colouring]
    Let $T$ be a triangulation of a convex polygon.  
    %a convex $(n+2)$-gon consisting of $n$ triangles and $C = \{c_1,c_2,\ldots, c_m\}$ be a set of colours. Let $F(T)$ denote the faces of $T$. 
    By a \emph{colouring} of $T$ we mean an assignment of colours from $1,\dots, m$ for every triangle of $T$. 
    %choice of colours assignment $C^T$ of $T$ we mean a map $F(T) \to C$. 
    %\todo{check that $C(T)$ or $C^T$ are not used anymore throughout}
\end{defn}

\begin{defn}[Coloured flip]
    Let  
    $C=\{1,\dots, m\}$ be a set of colours and $\sigma\in S_m$ be a permutation. Let $T$ be a triangulation of a convex polygon with each triangle a colour in $C$. 
    Let $t\in T$ be a diagonal incident with two triangles of the same colour $i$. Then the \emph{$\sigma$-flip} of $T$ at $t$ is defined as follows:  
    \begin{enumerate}
        \item Replace $t$ by the flip of $t'$ in the underlying uncoloured triangulation. 
        \item Change the colours of the two triangles incident with $t'$ to the colour $\sigma(i)$.  
    \end{enumerate}
    If the permutation $\sigma$ is a single cycle of the form $(1,2,\dots, m)$ (i.e. $i\mapsto i+1$), we call a $\sigma$-flip simply an {\em $m$-coloured flip}. 
\end{defn}

\begin{defn}
    Let $P$ be a convex polygon and let $C=\{1,\dots, m\}$ be a set of colours, let $\sigma\in S_m$ be a permutation. 
    The {\em coloured flip graph} of $P$ {\em with colours $C$ and permutation $\sigma$} or the 
    {\em $\sigma$-flip graph of $P$}
    is the graph whose vertices are the coloured triangulations of $P_n$ and whose edges correspond to $\sigma$-flips. 
    We will often just call it the {\em flip graph} of the polygon. 
\end{defn}

The coloured triangulations are also counted in terms of Catalan numbers. We study the coloured flip graphs in this paper. We note that whenever no two adjacent triangles have the same colour, no edge can be flipped and we have an isolated vertex in the flip graph. 

\begin{lem}
\label{lem:numberofcolouredtriangulations} 
    Consider a convex $n+2$-gon $P_{n+2}$ and a set $C$ of $m$ colours. 
    \begin{enumerate}[label=(\roman*)]
        \item There are $C_{n}m^{n}$ coloured triangulations of $P_{n+2}$,
        \item There are $C_{n}m(m-1)^{n-1}$ triangulations of $P_{n+2}$ where none of the diagonals can be flipped.
    \end{enumerate}
\end{lem}

\begin{proof}
    $(i)$. Any triangulation of $P_{n+2}$ has $n$ triangles, so there are $m^n$ different ways to colour a triangulation. The claim then follows from \cref{lem:catalannumbersgivenumberoftriangulations}. 

    $(ii)$. We consider the dual graph $G_T$ to a given triangulation $T$ of $P_{n+2}$: it has as vertices the triangles in $T$ and an edge between the two vertices of adjacent triangles. This graph is known to be a tree, it has $n$ vertices and at least two leaves. We start by colouring one leaf with one of the $m$ colours and then proceed to colour adjacent vertices. Since $G$ is a tree, for any new vertex we want to colour, there are $m-1$ options. Hence the factor $(m-1)^{n-1}$. The claim then follows with \cref{lem:catalannumbersgivenumberoftriangulations}. 
\end{proof}

\subsection{Triangulations of polygons and the Four-Colour Theorem}\label{sec:motivation}

The \emph{Four-Colour Theorem} is one of the most famous mathematical problems in history. It concerns the question whether four colours are enough to colour any map drawn in the plane. 
In 1977,  Appel, Haken and Koch 
established that four colours are enough (see \cite{10.1215/ijm/1256049011} and \cite{10.1215/ijm/1256049012}): 

\begin{thm}[Four-Colour Theorem]\label{thm:fourcolourthm}
    Any map on $\mathbb{R}^2$ can be colored using four colors such that any two regions sharing an edge are of different colours. 
\end{thm}

This result was proved with the assistance of computers. So far, there is no abstract proof of this theorem. The main motivation of this project is the search for an alternative 
approach to its proof. 
In 2002, Gravier and Payan showed that the Four-Colour Theorem is equivalent to a question on coloured flip graphs: 

\begin{conj}\label{conj1}
    Let $T_1,T_2$ be two arbitrary triangulations of a convex polygon $P$. Let $C=\{1,2\}$ be two colours. 
    %and $\sigma=(1,2)$ be the transposition of the two colours. 
    Then there exist colourings of $T_1$ and of $T_2$ such that there is a sequence of $2$-coloured flips between the two coloured triangulations. 
    %$C(T_1)$ and $C(T_2)$  such that there is a sequence of $2$-coloured flips from ${T_1}$ to ${T_2}$.
\end{conj}

\begin{thm}[\cite{GRAVIER2002817}]\label{thm1}
    Given any two triangulations of a convex polygon, it is possible to transform one into the other by a sequence of $2$-coloured flips if and only if 
    the Four-Colour Theorem holds. %every planar graph is 4-colourable.
\end{thm}

So in order to give an abstract approach to the Four-Colour Theorem, it is enough to give an 
abstract prove of \cref{conj1}.
We will recall the proof of \cref{thm1} in \cref{sec:4colour-signed}.

%%%%%%%%%%%%%%%
%
\section{Connected components of coloured flip graphs}\label{sec:components}

We first show that when studying coloured flip graphs, it is enough to consider $1$-coloured and $2$-coloured flips.

\begin{lem}\label{lm:one-cycle}
    It is enough to determine $\sigma$-flip graphs for single cycle permutations. 
\end{lem}
\begin{proof}
    If we have colours from multiple cycles, then we can divide the polygon up into smaller polygons, corresponding to areas with colours in the different cycles. These are invariant, as the cycle a colour comes from is invariant under a flip. Hence if we can get between permutations by using a colour permutation with multiple cycles, we can also get between them by using a single cycle with length the highest common multiple of all previous cycle lengths.
\end{proof}

So from now on we will assume that $\sigma$ is a single cycle of length $m$. We now show that it is enough to consider $m=1$ or $m=2$ depending on whether $m$ is odd or even. 

\begin{lem} 
%    Any colour permutation of flips determined by $\sigma\in S_n$ reduces to either the 1 or 2 colour case (since the 1-colour case is solved, we then only need to consider the 2 colour case).
%
Let $C=\{1,2,\dots,m\}$ be $m$ colours. To determine whether two 
triangulations $T_1$ and $T_2$ are linked by an $m$-coloured flip sequence, it is enough to consider the $1$-coloured case if $m$ is odd or the $2$-coloured case if $m$ is even. 
\end{lem}
\begin{proof}
    By Lemma~\ref{lm:one-cycle}, we can assume $\sigma$ is a singe cycle of length $m$:
    If $m$ is odd,  then colour every triangle in $T_1$ colour $1$, the first colour in the $m$-cycle. Note that after flipping the same quadrilateral $m$ times, the colour will still be the same but the diagonal is flipped. So by finding a path of uncoloured flips between $T_1$ and $T_2$, and replacing each move with $m$ moves on the same diagonal, we have found a path between them which respects the permutation (after each set of $m$ moves, the entire shape is still the same colour).
    If $m$ is even, and we can find a 2-coloured flip sequence between $T_1$ and $T_2$, then we can do so for any even cycle. Use the same colourings for $T_1$ and $T_2$, and find a path between them as follows: if in the original path we are going from colour 1 to colour 2, perform one flip. If we are going from colour 2 to colour 1, replace the single flip with m-1 flips, as this will change the colour and diagonal as we wanted.  
\end{proof}
An example is given in Figure \ref{fig:4col-to-2col}, where the yellow to red flip is replaced by 3 flips, and the red to yellow flip is left as a single flip.

%\begin{figure}[htbp]
%    \includegraphics[width = 0.7\textwidth]{Figures/image_2023-06-28_112019398.png}
%    \caption{an example sequence with one colour, translated into a sequence with 3 colours \\
%    ($\sigma = $ (red, yellow, cyan))}
%\end{figure}

\begin{figure}[htbp]
    \includegraphics[width = 0.7\textwidth]{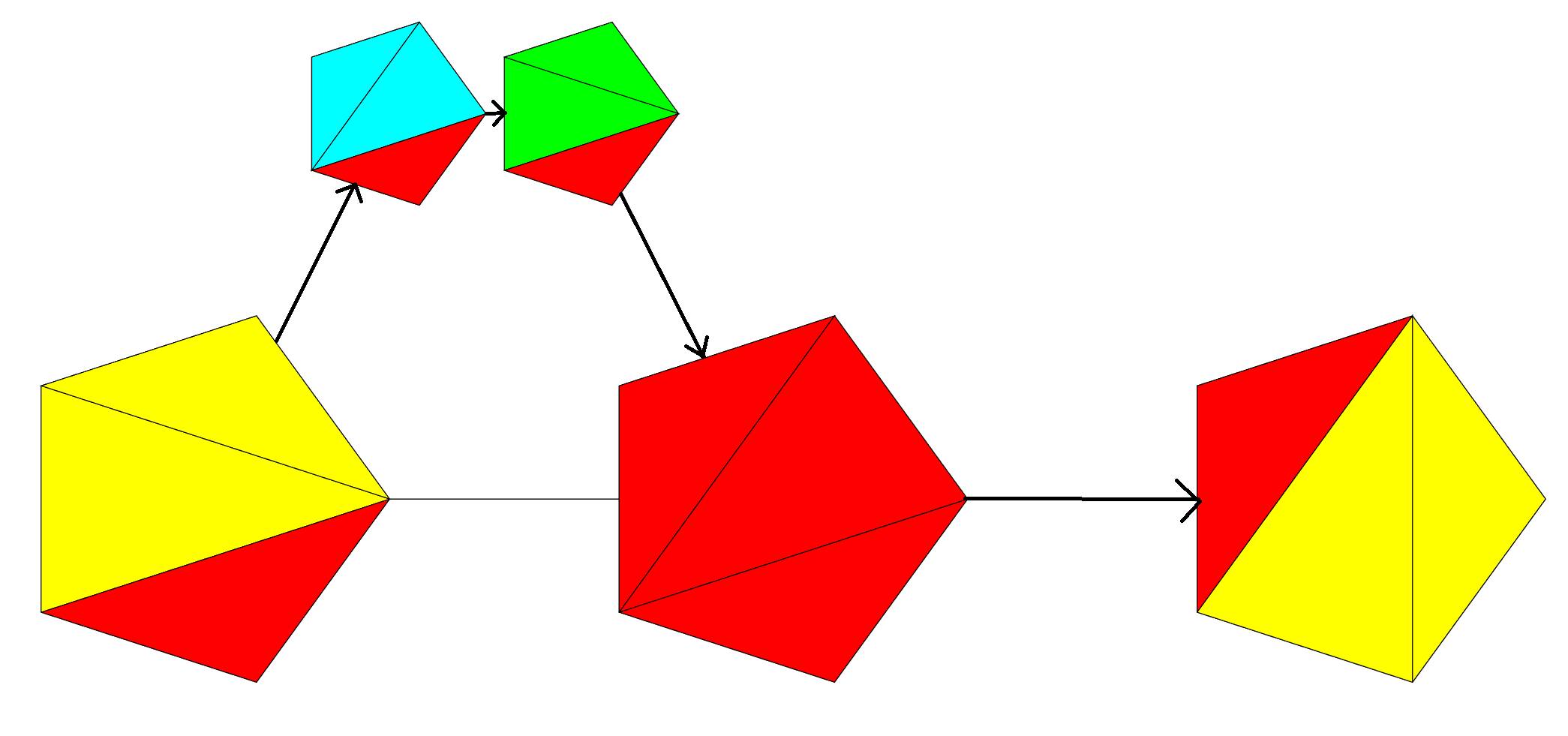}
    \caption{A flip sequence with two colours, translated into a sequence with four colours \\
    (with $\sigma = $ (red, yellow, cyan, green)) }\label{fig:4col-to-2col}
\end{figure}

\begin{exmp} 
We will again consider an example of coloured hexagon $P_6$ triangulations. In \cref{fig:connectedcomponentsofhexagon}, we list all types of connected components of the coloured flip graph for $P_6$.
    
    \begin{figure}[h!]
    \includegraphics[width = 0.45\textwidth]{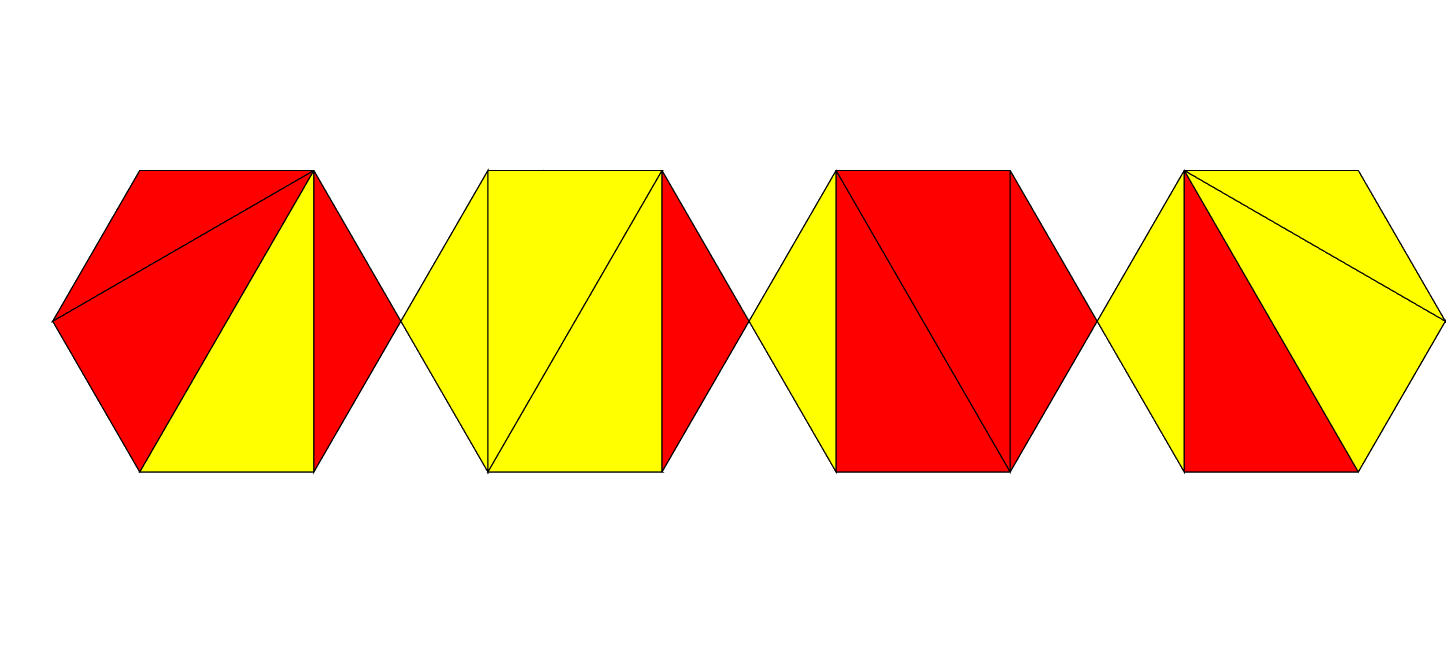}
    \includegraphics[width = 0.45\textwidth]{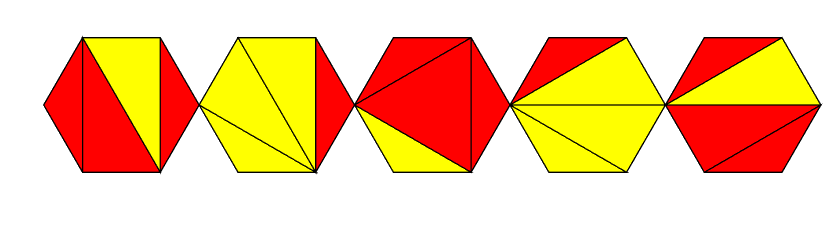}
    \includegraphics[width = 0.4\textwidth]{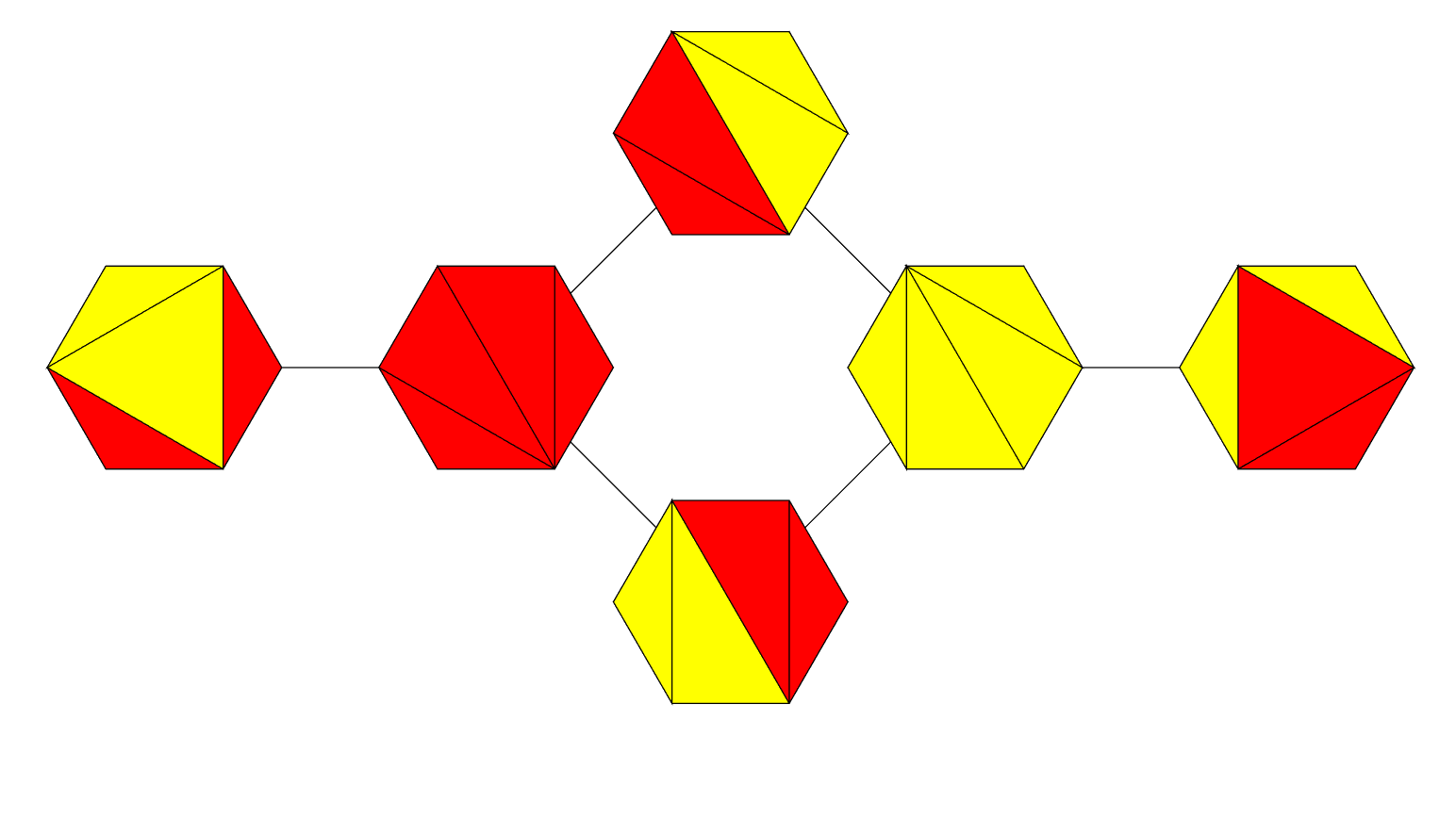}
    \includegraphics[width = .4\textwidth]{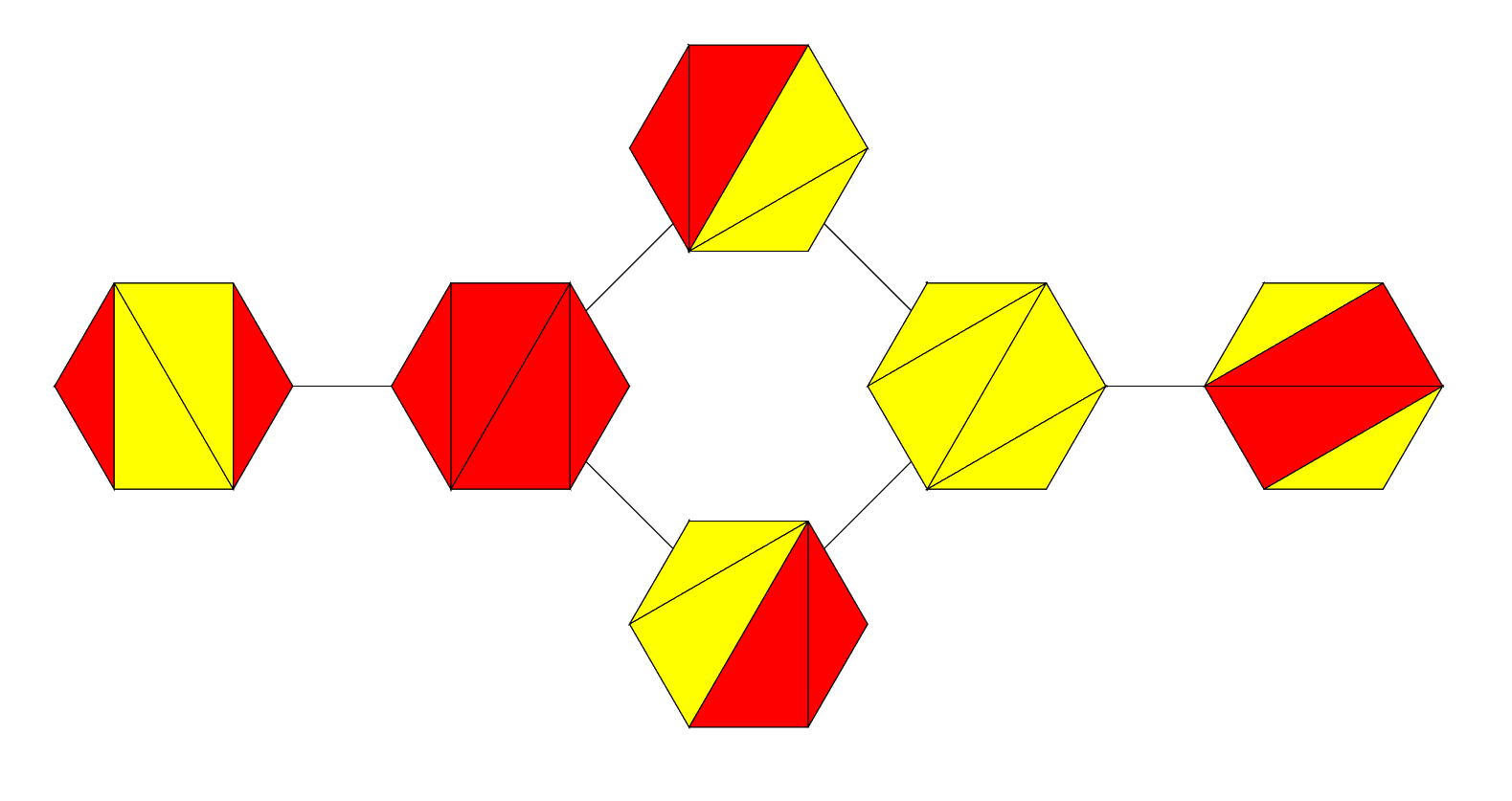}
    \includegraphics[width = .45\textwidth]{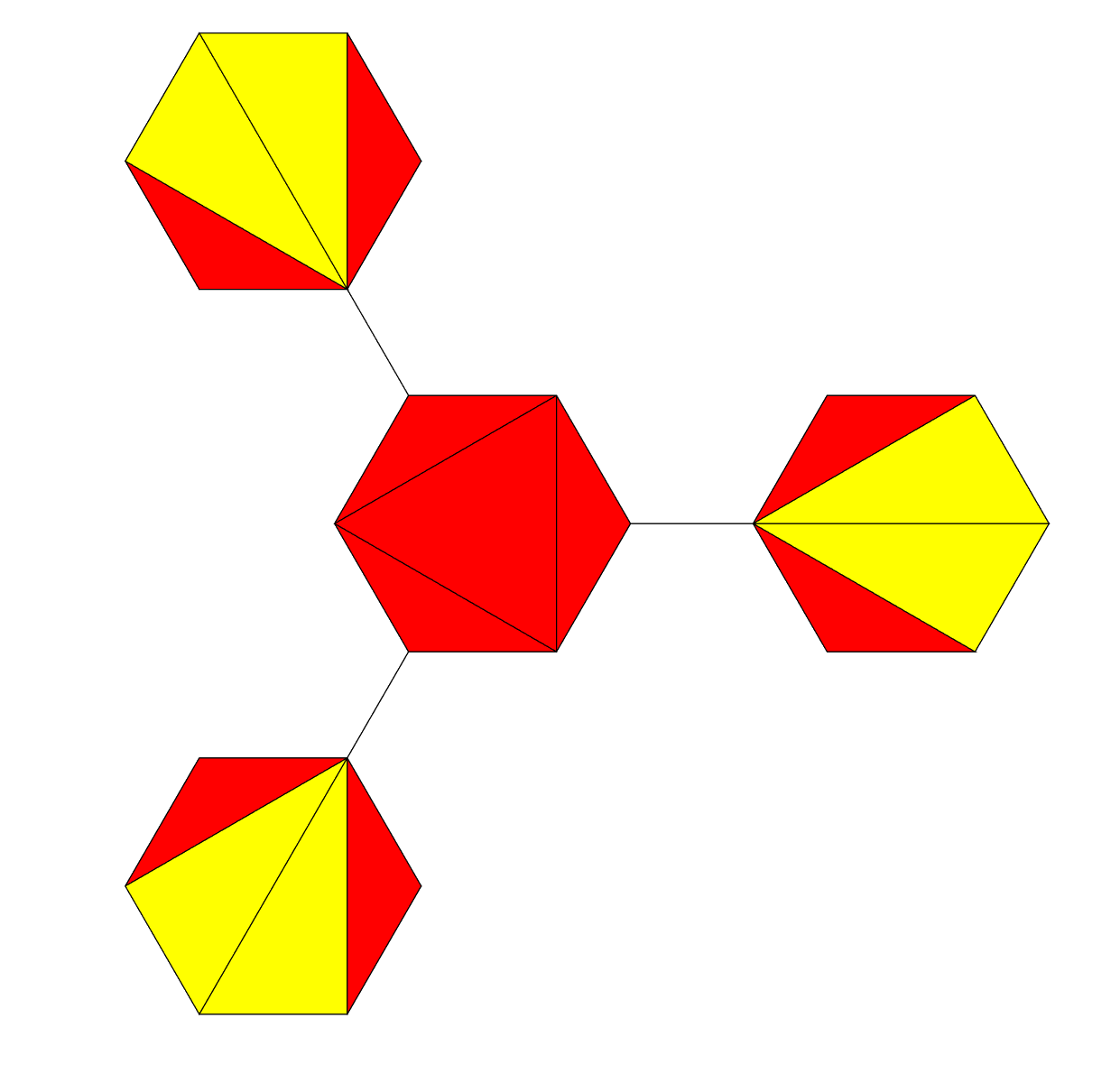}
    \caption{Connected components of the coloured flip graph of $P_6$ up to isometry, excluding the isolated points.}
    \label{fig:connectedcomponentsofhexagon}
\end{figure}

\end{exmp}

From now on, we restrict to $2$-coloured flips, i.e. to the case $m=2$. We will often choose $C=\{1,-1\}$ and indicate these colours by $+,-$ in the examples. We often use yellow and red as the two colours. 

\begin{defn}\label{def:flip-sequence}
For the following theorem, we need the notion of a flip sequence: let $T$ be a $2$-coloured triangulation of a convex polygon, let $s>0$. A {\em 2-coloured flip sequence} $\underline{\mu}=\mu_1\cdots \mu_s$ is a sequence of 2-coloured flips where for $i=s,s-1,\dots, 1$, there exists a diagonal in $\mu_{i-1}\mu_{i-2}\cdots\mu_s(T)$  which can be 2-colour flipped and where $\mu_i$ is a 2-coloured flip in the coloured triangulation $\mu_{i-1}\mu_{i-2}\cdots\mu_s(T)$. 
We denote the 2-coloured triangulation obtained through this sequence by $\underline{\mu}(T)$. 
\end{defn}

\begin{thm}\label{thm:component-size}
Let $G$ be the coloured flip graph of $P_{n+2}$. Then every connected component of $G$ is either an isolated point or is of size $\ge n$. 

Moreover, if $T$ is a triangulation in a non-trivial connected component and $t$ a diagonal of $T$, then either $t$ can be 2-coloured flipped or there exists a 2-coloured flip sequence 
$\underline{\mu}=\mu_1\cdots \mu_s$ (where $s\ge 1$) such that $t\in \underline{\mu}(T)$ and such that $t$ can be colour-flipped in $\underline{\mu}(T)$. 

% \kb{old version:}    Let $G$ be a coloured flip graph of triangulations of $P_{n+2}$. Then each connected component of $G$ which is not an isolated point has at least $n$ vertices, and in fact for all interior edges there is a sequence of moves which allows up to flip it.
\end{thm}
\begin{proof} 
Let $T$ be a triangulation of $P_{n+2}$, with a colouring. 
 We assume that there is at least one diagonal which can be 2-colour-flipped.  
We mark the diagonals of $T$ with a blue or a red dot: A diagonal $t$ is marked blue if it can be flipped after some (possibly empty, if it can be flipped immediately) sequence of coloured flips. Diagonals of $T$ which can never be flipped are marked with a red point. 
By assumption, at least one diagonal is marked with a blue dot. 
%     For the triangulation and colouring in question, put a blue dot on the diagonal if there is a sequence of moves allowing us to flip the edge, and red otherwise. If this is not an isolated point in it's coloured component, there is at least one blue dot (See \cref{fig:enter-label2} for an example arrangement). If there is also at least one red dot, find a triangle with a red dot on one side, and a blue dot on another (call the sides in question $B$ and $R$). 
If there exists a diagonal with a red dot, find a triangle with with a red and a blue dot on two of its diagonals. Such a triangle always exists (see~\cref{rem:why-red-blue}). 
%The connected component of $T$ is not a isolated point, so it has at least one diagonal marked blue and we assume that there is a diagonal marked red. Therefore, $n\ge 3$ and there exists one triangle which has one diagonal marked red and one diagonal marked blue, see \cref{rem:why-red-blue}. (\kb{still not very clear... but is it ok now?}). 
Call these two diagonals $B$ and $R$. The two triangles incident with $R$ must be of different colours, since otherwise, we can flip it immediately. 
We now execute a (coloured) flip sequence in order to flip the edge $B$. At some point in this sequence (or at the end of the sequence), the colour of one of the triangles incident with edge $R$ has changed colour, and at this point, the edge $R$ can be flipped. 

%((     During this process, the colour on one side of edge $R$ has changed, so there must exist a stage of this process when $R$ has the same colour on both sides of it, and hence can be flipped.)) 
This is a contradiction, hence there can be no red dots and every edge can be flipped eventually. Hence there are at least $n$ vertices in this connected component of the coloured  flip graph. 
\end{proof} 

\begin{remark}\label{rem:why-red-blue}
    Consider a triangulation $T$ where at least one diagonal can be 2-colour flipped. We mark the diagonals of $T$ by a blue dot if the diagonal can eventually be colour-flipped and with a red dot otherwise (as in the proof of \cref{thm:component-size}). Then there exists a triangle with a red and blue dot on two of its sides: We draw the dual graph $G_T$ of the triangulation (as in the proof of Lemma~\ref{lem:numberofcolouredtriangulations}): 
    it has vertices for the diagonals of $T$ and edges between vertices $t,t'$ whenever there is a triangle containing $t$ and $t'$. We equip this graph with the two colours. Since the graph $G_T$ is connected, there has to be an edge between a red and a blue node.  

    \vskip.2cm

%    \begin{figure}[h!]
%    \centering

%    \caption{An example on a nonagon, red represents edges which are in every triangulation in the connected component, blue means the edge can be flipped by some sequence of moves}
%    \label{fig:enter-label2}
%\end{figure}
\end{remark}

%\begin{rem}
%    The graph has size at least $n$, since if it has size less than that there must be a cycle consisting of flips of distinct edges which results in no change, however then one of the edges must be flipped exactly once and so we cannot end up with the same triangulation. 
%\end{rem}

The following result shows that connected components of size $n$ do exist in the coloured flip graph of $P_{n+2}$. 

\begin{exmp}\label{ex:conn-comp-fan}
Consider a fan triangulation $T$, with alternatingly coloured triangles apart from at one end. See \cref{fig:fan-line} for an illustration of such a fan triangulation of a decagon. 
Then the connected component containing $T$ is a single line with $n$ vertices: in every step, only a single 2-coloured flip can be made. Under this, the two triangles of the same colour move from one side of the fan to the other end of the fan. 
\end{exmp}
\begin{figure}[h!]
    \centering
    \includegraphics[width = .28\textwidth]{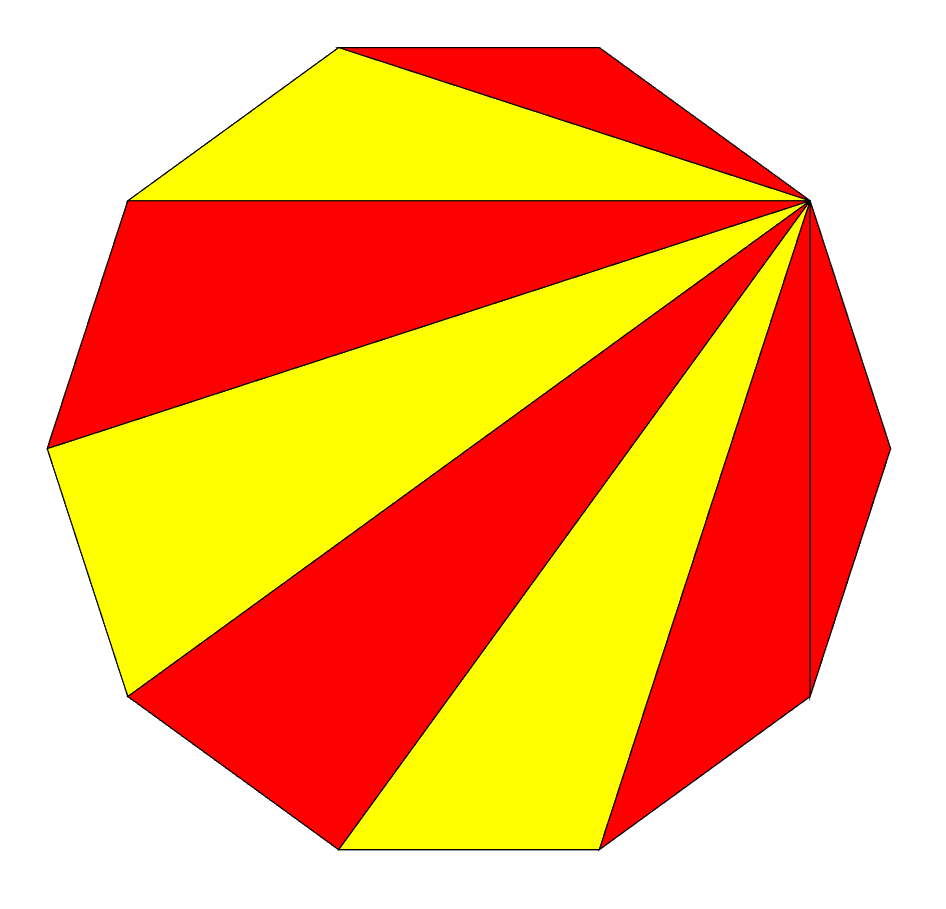}
    \caption{A coloured fan triangulation of a decagon.}\label{fig:fan-10gon}
\end{figure}

We recall the notion of a weighting on the vertices of a triangulated polygon with a $2$-colouring from~\cite{GRAVIER2002817}. 

If we have a coloured triangulation and if $\triangle$ is a triangle of $T$, we denote its colour by 
$s(\triangle)$. 
Recall that 
$s(\triangle)\in\{-1,+1\}$. 

\begin{defn}\label{def:weighting}
A {\em weighting} of the polygon $P$ is given by a choice of a triangulation $T$ and a function 
$p$ assigning to each vertex of $P$ an element of $\{-1,0,1\}$ such that there is a $2$-colouring of $T$ where for every vertex $x$ of $P$, we have 
$p(x)=\sum_{x \in\triangle} s(\triangle) \mod 3$ (the sum is taken over all triangles incident with the vertex $x$). 
\end{defn}

% By definition, any 2-colouring of a triangulated polygon determines a weighting. 
% ((\kb{then remove the following:} Consider a corner $x$ of a polygon which is coloured by assigning colours $+ = +1$ or $- = -1$ to each of it's sides. Define the value of $x$ as the values of each triangle with a vertex at $x$, taken mod 3 (so $ 1 = +1 = +$ and $2 = -1 = -$). 
%     If the value is given at every corner, we call this a valuation.))
% %\end{defn}

Two weightings of coloured triangulated quadrilaterals are shown in~\cref{fig:exampleofquadvaluation}. 

\begin{lem}
    Any two $2$-coloured triangulations which are in the same connected component of the flip graph have the same weighting. 
\end{lem}
\begin{proof}
    Let $p$ be a weighting of a triangulated polygon with 2-colouring. If $x$ is a vertex of the quadrilateral where the flip happens, the flip either changes two triangles with $+1$ to one triangle with $-1$ (or vice versa) or two triangles with $-1$ to one triangle with $+1$ (or vice versa). In all cases, $p(x)$ remains the same. (See~\cref{fig:exampleofquadvaluation}). 
\end{proof}

However, there exist 2-colourings which are not flip equivalent but have the same weighting, see Figure~\ref{fig:weights-not-flip} for an example.

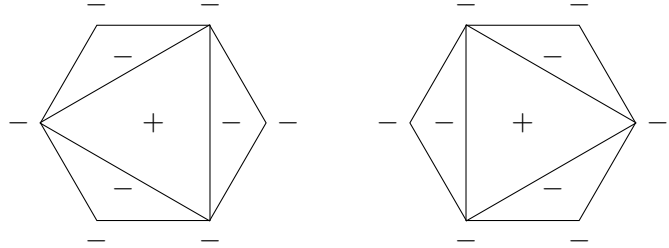
\begin{figure}[h!]
    \centering
        \begin{tikzpicture}[scale=.8]
            \node[draw, regular polygon, regular polygon sides=6, minimum size=3cm] (a) {};
            \draw (a.corner 1) edge (a.corner 3);
            \draw (a.corner 3) edge (a.corner 5);
            \draw (a.corner 1) edge (a.corner 5);
            \node[above] at (a.corner 1) {$-$};
            \node[above] at (a.corner 2) {$-$};
            \node[left] at (a.corner 3) {$-$};
            \node[below] at (a.corner 4) {$-$};
            \node[below] at (a.corner 5) {$-$};
            \node[right] at (a.corner 6) {$-$};
            \node at (0,0) {$+$};
            \node at (-.5,1.1) {$-$};
            \node at (1.3,0) {$-$};
            \node at (-.5,-1.1) {$-$};
        \end{tikzpicture}\qquad
        \begin{tikzpicture}[scale=.8]
            \node[draw, regular polygon, regular polygon sides=6, minimum size=3cm] (a) {};
            \draw (a.corner 2) edge (a.corner 4);
            \draw (a.corner 4) edge (a.corner 6);
            \draw (a.corner 2) edge (a.corner 6);
            \node[above] at (a.corner 1) {$-$};
            \node[above] at (a.corner 2) {$-$};
            \node[left] at (a.corner 3) {$-$};
            \node[below] at (a.corner 4) {$-$};
            \node[below] at (a.corner 5) {$-$};
            \node[right] at (a.corner 6) {$-$};
            \node at (0,0) {$+$};
            \node at (.5,1.1) {$-$};
            \node at (-1.3,0) {$-$};
            \node at (.5,-1.1) {$-$};
        \end{tikzpicture}
    \caption{Two triangulations which are not flip equivalent but have the same weighting}
    \label{fig:weights-not-flip}
\end{figure}

The following statement is mentioned in~\cite{GRAVIER2002817}. We include a proof for completeness. 

\begin{thm} 
    Given a weighting of a triangulation, there is at most one way to colour it to match the weighting. 
\end{thm}
\begin{proof}
    Since any triangulation must contain a triangle with 2 of it's sides being sides of the polygon, there is a vertex which is only contained in one triangle. At this vertex, if the value is zero then there is no such colouring, and if it is $1$ or $2$ then this determines the colour of the triangle. Consider removing this triangle, and subtracting off the value it contributes to the neighbouring triangles to give a valuation and triangulation for a $(n-1)$-gon, repeat until we either reach a contradiction or have completely coloured the shape. Hence if the colouring exists, it must be unique. (see \cref{fig:exampleofquadvaluation})
\end{proof}

\begin{figure}[h!]
    \centering
    \begin{tikzpicture}[baseline]
        \node[draw, regular polygon, regular polygon sides=4, minimum size=3cm] (a) {};
        \foreach \i in {1,2,3,4}
            \node[circle, draw, fill=white, fill opacity=1, scale=0.3] at (a.corner \i) {};
        \draw (a.corner 1) edge (a.corner 3);
        \node[above] at (a.corner 1) {$-$};
        \node[above] at (a.corner 2) {$+$};
        \node[below] at (a.corner 3) {$-$};
        \node[below] at (a.corner 4) {$+$};
        \node at (-.5,.5) {+};
        \node at (.5,-.5) {+};
    \end{tikzpicture}\qquad
    \begin{tikzpicture}
        \draw [->, thick] (2,6) -- (4,6) node [pos=0.5,above,font=\footnotesize] {flip};
    \end{tikzpicture}\qquad
    \begin{tikzpicture}[baseline]
        \node[draw, regular polygon, regular polygon sides=4, minimum size=3cm] (a) {};
        \foreach \i in {1,2,3,4}
            \node[circle, draw, fill=white, fill opacity=1, scale=0.3] at (a.corner \i) {};
        \draw (a.corner 2) edge (a.corner 4);
        \node[above] at (a.corner 1) {$-$};
        \node[above] at (a.corner 2) {$+$};
        \node[below] at (a.corner 3) {$-$};
        \node[below] at (a.corner 4) {$+$};
        \node at (-.5,-.5) {-};
        \node at (.5,.5) {-};
    \end{tikzpicture}
    \caption{Example of quadrilateral valuation}
    \label{fig:exampleofquadvaluation}
\end{figure}
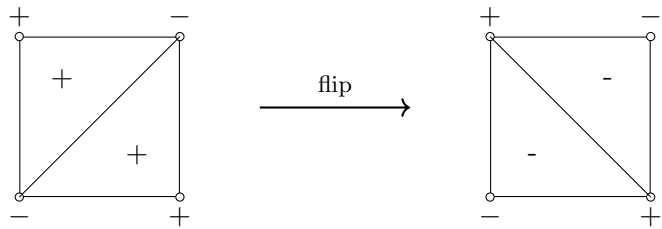

\begin{thm}\label{thm:cycles-are-even}
    Any cycle in the coloured flip graph is even. 
%    Any connected component of the coloured flip graph contains no odd cycles. 
\end{thm}
\begin{proof}
Let $T$ be a 2-coloured triangulation, let $X$ be the number of triangles marked $+$ in this triangulation. After every flip, $X$ either increases or decreases by 2, hence every flip changes $X$ by $\pm 2$. 
Therefore, if we reach $T$ again after a sequence of coloured flips, this sequence has to have even length, since the number of triangles marked with a $+$ will be equal to $X$ again. 

\end{proof}

\subsection{Structure of the flip graph}

In this section, we show properties of the flip graph. In particular, we prove the existence of hypercubes in the flip graph. 

\begin{defn}\label{def:independent-flips}
    Let $t,t'$ be two diagonals in a triangulation of a convex polygon. If two quadrilaterals share at most a diagonal, we say that the quadrilaterals are {\em disjoint}. In this case, we say that the flips of $t$ and of $t'$ are {\em independent}. 
    %We consider two quadrilaterals in a triangulation to be disjoint if their intersection has zero area. We call flips of the diagonals of these quadrilaterals as \emph{independent flips}.
\end{defn}

\begin{exmp}\label{ex:size-fan}
    Let $T$ be a fan triangulation of $P_{n+2}$ where each face is assigned the same colour (see Figure~\ref{fig:fan-10gon} for an example). 
    
    % Then for the number of independent flips $k$ we have $\lfloor \frac{n+1}{3} \rfloor \leq k \leq \lfloor \frac{n}{2} \rfloor$. 
    % \kb{Q6: should it not be: ``is equal to $\lfloor\frac{n}{2}\rfloor$''??? what is the lower bound???}

    Notice that we have $n$ triangles in $T$ for a $(n+2)$-gon. To maximise $k$, we start by choosing the first two triangles on the left in $T$. Then we keep choosing the quadrilateral close to the previous one sharing a common boundary. We end up with either all triangles been chosen or having one left. Hence, $\lfloor \frac{n}{2} \rfloor$ is the maximum.
        
%    To minimise $k$, we group every 3 triangles together into a pentagon. We ignore the first triangle, and consider the latter two constituting the quadrilateral we need. Similarly, we start from the left in $T$, we either end up with no triangles, only one triangle, or two triangles. We have another disjoint quadrilateral only when there are two triangles at last. Hence the minimum bound is $\lfloor \frac{n+1}{3} \rfloor$.

\end{exmp}

\begin{figure}[h!]
    \centering
    \includegraphics[width =.28\textwidth]{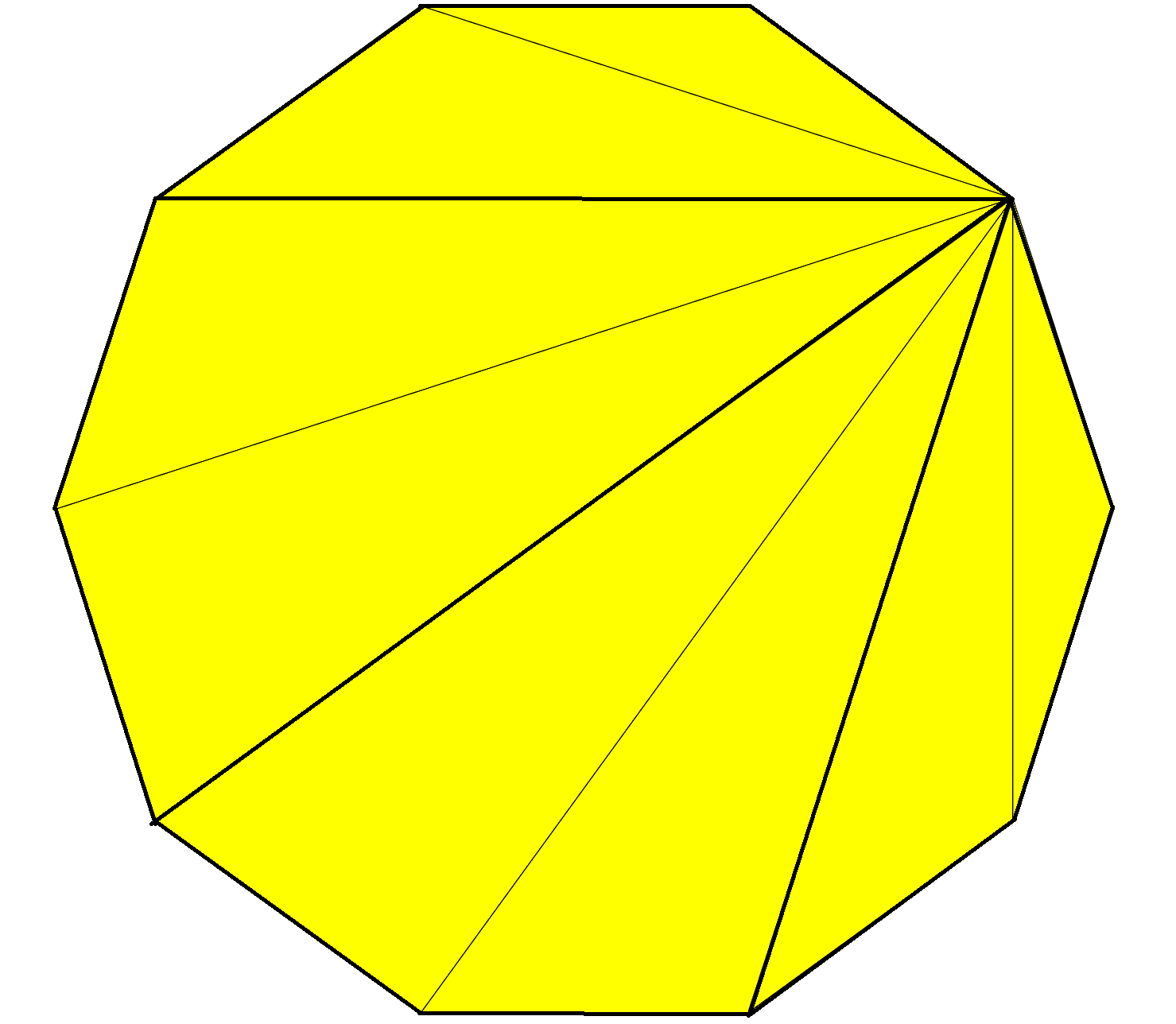}
    \caption{A flip graph of the monocolour $n+2$-gon, $n \geq 8$ contains 4 non-overlapping quadrilaterals, and hence a 4-dimensional hypercube subgraph.}\label{fig:fan-line}
\end{figure}

\begin{prop}
\label{prop:hypercubesinflipgraphsareindependentflips}
Let $T$ be a 2-coloured triangulation of a convex polygon. 
Let $G$ be the connected component of the flip graph containing $T$. 
Assume that there are $k>1$ diagonals in $T$ which can be 2-coloured flipped and whose  quadrilaterals are pairwise disjoint. 
Then $G$ contains a $k$-dimensional hypercube. 

%\kb{then remove this:}
%    Suppose $C^T$ is a coloured triangulation of a convex $(n+2)$-gon. If there are $k$ (where $k \geq 2$) disjoint quadrilaterals in which flips are possible, then the connected component containing $C^T$ contains a $k$-dimensional hypercube.
\end{prop}

\begin{proof}
    Denote that $k$ diagonals of $T$ which can be flipped independently by $1,2,\dots, k$. For any $i\ne j$, $1\le i,j\le k$, the flips $\mu_{i}$ and $\mu_{j}$ commute. We consider all the triangulations which can be reached from $T$ by arbitrary $2$-coloured flips of these $k$ diagonals. In the subgraph of the flip graph they define, each of them has degree $k$. So they form a subgraph isomorphic to a $k$-dimensional hypercube as claimed. 
\end{proof}

\begin{coroll}\label{cor:not-planar}
For $n\ge 8$, the 2-coloured flip graph of $P_{n+2}$ contains a connected component which is not planar. 
\end{coroll}

\begin{proof} 
    Consider the fan triangulation where every triangle is coloured with the same colour. Let $G$ be the connected component of the coloured flip which contains this coloured fan triangulation. Since $n\ge 8$, there is a sub-polygon with the same structure as the fan decagon (see figure 8). Hence there are at least four quadrilaterals which can be flipped independently, given by the thick lines. Therefore, $G$ contains a 4-dimensional hypercube by \cref{prop:hypercubesinflipgraphsareindependentflips}. Denote this by $Q_4$. Since $Q_4$ has the complete bipartite graph $K_{3,3}$ as a subgraph, and the latter is not planar, $G$ cannot be planar.

%    of the flip graph independent 
%    coloured fan triangulation $C^F$ of an decagon where each face is coloured the same colour. The triangulation has 7 diagonals and it is possible to choose four of the diagonals so that the four flips are independent. Then, by \cref{prop:hypercubesinflipgraphsareindependentflips}, the flips give rise to a 4-dimensional hypercube, denoted by $Q_4$, in the connected component containing $C^F$. Since $Q_4$ has a subgraph which is a subdivision of $K_{3,3}$, the connected component cannot be planar. Furthermore, $C^F$ is contained in some coloured triangulations of $(n+2)$-gons ($n > 8$) so the coloured flip graphs also contain $Q_4$ and hence are not planar.
\end{proof}

\begin{note} 
    We consider two $k$-dimensional hypercubes in a connected component of the flip graph to be {\em distinct} if they are disjoint or if their intersection is a union of hypercubes of smaller dimension. 
    %share a $m$-dimensional face %(\kb{I don't think we have defined the ``faces'' of a hypercube... we should do so or rephrase this}), where $0\le m \leq k-1$ or if they are disjoint. 
% as distinct if they share a $m$-dimensional face, where $m \leq k-1$, or have an empty intersection.
\end{note}

\begin{lem}\label{lm:hypercube}
    Suppose $T$ that is a fan triangulation of a convex $(n+2)$-gon where all triangles have the same colour. Let $G$ be a connected component of the coloured flip graph. Then
    \begin{enumerate}[label=(\roman*)]
        \item if $n$ is even, then $G$ contains a $\frac{n}{2}$-dimensional hypercube and a $(\frac{n}{2}-1)$-dimensional hypercube; 
        % \kb{we can say a bit more... number of such is at least...(maybe at least $\frac{n}{2}+1$?)}
        
        %Diana: I think the flips might give the same hypercube since there are only two different ways to group the faces?
        
        %then the connected component of the flip graph containing $C^T$ has a $\frac{n}{2}$-dimensional hypercube and a ($\frac{n}{2} - 1$)-dimensional hypercube;
        \item if $n$ is odd, then $G$ contains at least two $(\frac{n-1}{2})$-dimensional hypercube.

        % if $n$ is odd, 
        % $G$ contains at least $\frac{n-1}{2}$ distinct $(\frac{n-1}{2}-1)$-dimensional hypercubes (\kb{should it be at least $\frac{n+1}{2}$ hypercubes of dimension $\frac{n-1}{2}$??}). 
        %Diana: I believe this depends on which face we leave out - this could point to a recursive formula
        
        %then the connected component of the flip graph containing $C^T$ has at least $\frac{n-1}{2}$ distinct $(\frac{n-1}{2})$-dimensional hypercubes, 
    \end{enumerate}
\end{lem}

\begin{proof}
    % We proceed by induction on $n$. As a base case, consider a fan triangulation of a pentagon.

    \begin{enumerate}[label=(\roman*)]
        \item When we assume $n$ to be even, the maximum number of independent coloured flips is $\frac{n}{2}$, where we group pairs of adjacent faces starting with a face defined by a single diagonal. In a similar way, if we group adjacent faces starting from a face defined by two diagonals there are two possibilities. Either we end up with $\frac{n}{2}$ independent flips, or we have $\frac{n}{2}-1$ independent flips.
        % Since $n$ is even, then $\max k = \frac{n}{2}$. Also, it is always possible to contain a hypercube of dimension $k=\frac{n}{2}-1$ by removing the first and last quadrilateral.
        \item Similar to (i) $\max k = \frac{n-1}{2}$ by removing either the first or the last triangle. Hence, we have at least two distinct hypercubes of dimension $k = \frac{n-1}{2}$. 
    \end{enumerate}
\end{proof}

Note that a version of \cref{lm:hypercube} can be proved for more general triangulations: the number of disjoint quadrilaterals in an arbitrary 2-coloured triangulations gives a lower bound on the dimension of maximal dimensional hypercubes it contains. However, it is more difficult to determine the number of independent quadrilaterals, especially if there are inner triangles.

% Notice that the statement of \cref{lm:hypercube} does not only work for fan triangulations. For triangulations without triangles in its inner faces whose edges are all diagonals, the hypercube graph is of the same shape, since we can choose the disjoint quadrilaterals in the same way. But if such triangle exists in $T$, then there are more than two vertices of degree 2 in $T$. Our choice of disjoint quadrilaterals is limited.

%Moreover, the lower bound of the number of hypercubes in a connected component containing this fan triangulation is determined by its number of possible choices of disjoint quadrilaterals. It is the lower bound since the resulting triangulation obtained in hypercubes is not the endpoint of the flip graph. Every choice of $k$ disjoint quadrilaterals corresponds to a $k$-dimensional hypercube (see \cref{exmp:hypercubeconj}).

\begin{exmp}\label{exmp:hypercubeconj}
We illustrate~\cref{lm:hypercube} for the uni-colored fan triangulation of $P_{n}$ with $6\le n\le 9$ in Figures~\ref{fig:conj5hexagon}, \ref{fig:conj5heptagon}, \ref{fig:conj5octagon} and~\ref{fig:conj5nonagon}. In each figure from left to right, the first graph is the original triangulation, and then are the possible hypercubes of different dimensions, and the last one is the combination of all these hypercubes. We number the diagonals in $T$, and the numbers on edges of the $k$-dimensional cube represents a flip of that diagonal. The black points represent the fan triangulation.
\end{exmp}

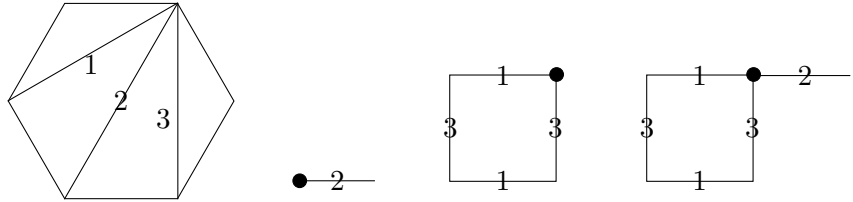
\begin{figure}[h!]
    \centering
    \begin{tikzpicture}[scale=.8]
        \node[draw, regular polygon, regular polygon sides=6, minimum size=3cm] (a) {};
        \draw (a.corner 1) edge (a.corner 3);
        \draw (a.corner 1) edge (a.corner 4);
        \draw (a.corner 1) edge (a.corner 5);
        \node at (-.5,.6) {1};
        \node at (0,0) {2};
        \node at (.7,-.3) {3};
    \end{tikzpicture}\qquad
    \begin{tikzpicture}
        \node[circle, draw, fill=black, scale=.5] at (0,0) {};
        \draw (0,0)--(1,0);
        \node at (.5,0) {2};
    \end{tikzpicture}\qquad
    \begin{tikzpicture}
        \node[draw, regular polygon, regular polygon sides=4, minimum size=2cm] (a) at (2,0) {};
        \node[circle, draw, fill=black, scale=.5] at (a.corner 1) {};
        \node at (1.3,0) {3};
        \node at (2.7,0) {3};
        \node at (2,.7) {1};
        \node at (2,-.7) {1};
    \end{tikzpicture}\qquad
    \begin{tikzpicture}
        \node[draw, regular polygon, regular polygon sides=4, minimum size=2cm] (a) at (2,0) {};
        \draw (2.7,.7)--(4,.7);
        \node[circle, draw, fill=black, scale=.5] at (a.corner 1) {};
        \node at (1.3,0) {3};
        \node at (2.7,0) {3};
        \node at (2,.7) {1};
        \node at (2,-.7) {1};
        \node at (3.4,.7) {2};
    \end{tikzpicture}
    \caption{The connected component containing the fan triangulation of hexagon has the middle two $1$-dimensional and $2$-dimensional cubes}
    \label{fig:conj5hexagon}
\end{figure}

\begin{figure}[h!]
    \centering
    \begin{tikzpicture}
        \node[draw, regular polygon, regular polygon sides=7, minimum size=3cm] (a) {};
        \draw (a.corner 1) edge (a.corner 3);
        \draw (a.corner 1) edge (a.corner 4);
        \draw (a.corner 1) edge (a.corner 5);
        \draw (a.corner 1) edge (a.corner 6);
        \node at (-.7,.6) {1};
        \node at (-.3,.1) {2};
        \node at (.3,.1) {3};
        \node at (.7,.6) {4};
    \end{tikzpicture}\qquad
    \begin{tikzpicture}
        \node[draw, regular polygon, regular polygon sides=4, minimum size=2cm] (a) at (2,0) {};
        \node[draw, regular polygon, regular polygon sides=4, minimum size=2cm] (b) at (3.4,0) {};
        \node[draw, regular polygon, regular polygon sides=4, minimum size=2cm] (c) at (3.4,1.4) {};
        \node[circle, draw, fill=black, scale=.5] at (a.corner 1) {};
        \node at (3.4,.7) {4};
        \node at (3.4,-.7) {4};
        \node at (3.4,2.2) {4};
        \node at (1.3,0) {1};
        \node at (2.7,0) {1};
        \node at (4.2,0) {1};
        \node at (2.7,1.4) {2};
        \node at (4.2,1.4) {2};
        \node at (2,.7) {3};
        \node at (2,-.7) {3};
    \end{tikzpicture}
    \caption{The connected component containing the fan triangulation of heptagon has the $2$-dimensional cubes shown on the right}
    \label{fig:conj5heptagon}
\end{figure}
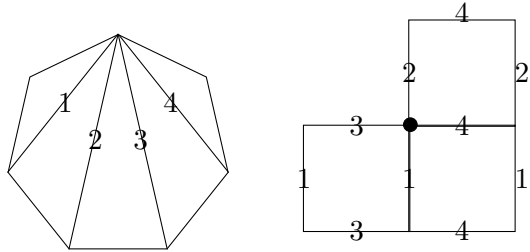

\begin{figure}[h!]
    \centering
    \begin{tikzpicture}
        \node[draw, regular polygon, regular polygon sides=8, minimum size=3cm] (a) {};
        \draw (a.corner 1) edge (a.corner 3);
        \draw (a.corner 1) edge (a.corner 4);
        \draw (a.corner 1) edge (a.corner 5);
        \draw (a.corner 1) edge (a.corner 6);
        \draw (a.corner 1) edge (a.corner 7);
        \node at (-.5,1) {1};
        \node at (-.5,.3) {2};
        \node at (-.1,-.3) {3};
        \node at (.6,0) {4};
        \node at (1,.4) {5};
    \end{tikzpicture}\qquad
    \begin{tikzpicture}
        \node[draw, regular polygon, regular polygon sides=4, minimum size=2cm] (a) at (2,0) {};
        \node[draw, regular polygon, regular polygon sides=4, minimum size=2cm] (b) at (3.4,0) {};
        \node[draw, regular polygon, regular polygon sides=4, minimum size=2cm] (c) at (3.4,1.4) {};
        \node[circle, draw, fill=black, scale=.5] at (a.corner 1) {};
        \node at (3.4,.7) {2};
        \node at (3.4,-.7) {2};
        \node at (3.4,2.2) {2};
        \node at (1.3,0) {4};
        \node at (2.7,0) {4};
        \node at (4.2,0) {4};
        \node at (2.7,1.4) {5};
        \node at (4.2,1.4) {5};
        \node at (2,.7) {1};
        \node at (2,-.7) {1};
    \end{tikzpicture}\qquad
    \begin{tikzpicture}
        \node[circle, draw, fill=black, scale=.5] at (0,2,0) {};
        \draw[thick](2,2,0)--(0,2,0)--(0,2,2)--(2,2,2)--(2,2,0)--(2,0,0)--(2,0,2)--(0,0,2)--(0,2,2);
        \draw[thick](2,2,2)--(2,0,2);
        \draw[dashed](2,0,0)--(0,0,0)--(0,2,0);
        \draw[dashed](0,0,0)--(0,0,2);
        \draw(1,2,2) node{3};
        \draw(1,2,0) node{3};
        \draw(1,0,2) node{3};
        \draw(1,0,0) node{3};
        \draw(0,1,2) node{5};
        \draw(0,1,0) node{5};
        \draw(2,1,0) node{5};
        \draw(2,1,2) node{5};
        \draw(0,2,1) node{1};
        \draw(0,0,1) node{1};
        \draw(2,0,1) node{1};
        \draw(2,2,1) node{1};
    \end{tikzpicture}\qquad
    \begin{tikzpicture}
        \node[circle, draw, fill=black, scale=.5] at (0,2,0) {};
        \draw[thick](2,2,0)--(0,2,0)--(0,2,2)--(2,2,2)--(2,2,0)--(2,0,0)--(2,0,2)--(0,0,2)--(0,2,2);
        \draw[thick](2,2,2)--(2,0,2);
        \draw[dashed](2,0,0)--(0,0,0)--(0,2,0);
        \draw[dashed](0,0,0)--(0,0,2);
        \draw[thick](0,2,0)--(0,4,0)--(0,4,2)--(0,2,2);
        \draw[thick](0,4,0)--(0,4,-2)--(0,2,-2)--(0,2,0);
        \draw[thick](0,2,-2)--(0,0,-2)--(0,0,0);
        \draw(1,2,2) node{3};
        \draw(1,2,0) node{3};
        \draw(1,0,2) node{3};
        \draw(1,0,0) node{3};
        \draw(0,1,2) node{5};
        \draw(0,1,0) node{5};
        \draw(2,1,0) node{5};
        \draw(2,1,2) node{5};
        \draw(0,1,-2) node{5};
        \draw(0,2,1) node{1};
        \draw(0,0,1) node{1};
        \draw(2,0,1) node{1};
        \draw(2,2,1) node{1};
        \draw(0,4,1) node{1};
        \draw(0,2,-1) node{2};
        \draw(0,0,-1) node{2};
        \draw(0,4,-1) node{2};
        \draw(0,3,2) node{4};
        \draw(0,3,0) node{4};
        \draw(0,3,-2) node{4};
    \end{tikzpicture}
    \caption{The connected component containing the fan triangulation of octagon has the middle two $2$-dimensional and $3$-dimensional cubes}
    \label{fig:conj5octagon}
\end{figure}
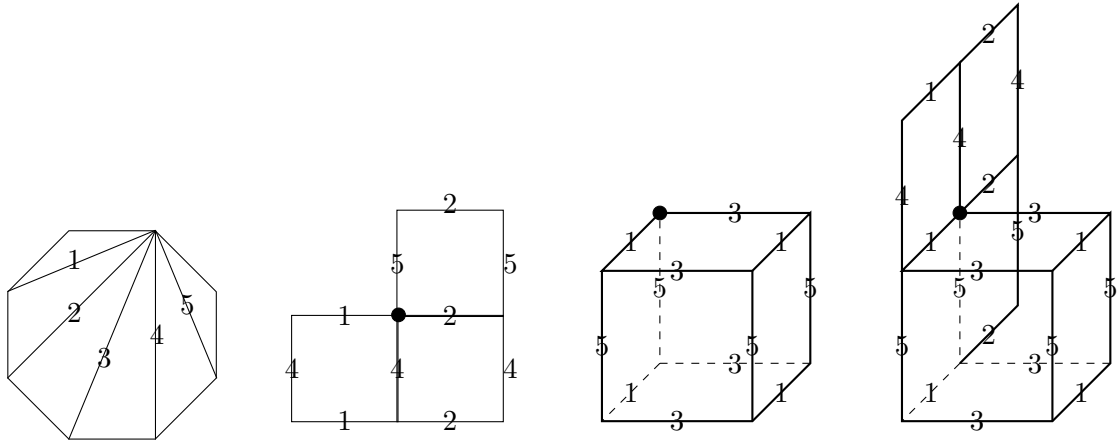

\begin{figure}[h!]
    \centering
    \begin{tikzpicture}[scale=.8]
        \node[draw, regular polygon, regular polygon sides=9, minimum size=3cm] (a) {};
        \draw (a.corner 1) edge (a.corner 3);
        \draw (a.corner 1) edge (a.corner 4);
        \draw (a.corner 1) edge (a.corner 5);
        \draw (a.corner 1) edge (a.corner 6);
        \draw (a.corner 1) edge (a.corner 7);
        \draw (a.corner 1) edge (a.corner 8);
        \node at (-1,1) {1};
        \node at (-.7,.3) {2};
        \node at (-.3,-.3) {3};
        \node at (.3,-.3) {4};
        \node at (.8,.3) {5};
        \node at (1,1) {6};
    \end{tikzpicture}\qquad
    \begin{tikzpicture}
        \node[draw, regular polygon, regular polygon sides=4, minimum size=2cm] (a) at (2,0) {};
        \node[circle, draw, fill=black, scale=.5] at (a.corner 1) {};
        \node at (1.3,0) {5};
        \node at (2.7,0) {5};
        \node at (2,.7) {2};
        \node at (2,-.7) {2};
    \end{tikzpicture}\qquad
    \begin{tikzpicture}[scale=.8]
        \node[circle, draw, fill=black, scale=.5] at (0,2,0) {};
        \draw[thick](2,2,0)--(0,2,0)--(0,2,2)--(2,2,2)--(2,2,0)--(2,0,0)--(2,0,2)--(0,0,2)--(0,2,2);
        \draw[thick](2,2,2)--(2,0,2);
        \draw[dashed](2,0,0)--(0,0,0)--(0,2,0);
        \draw[dashed](0,0,0)--(0,0,2);
        \draw[thick](0,2,0)--(0,4,0)--(0,4,2)--(0,2,2);
        \draw[thick](0,4,0)--(0,4,-2)--(0,2,-2)--(0,2,0);
        \draw[dashed](0,2,0)--(-2,2,0)--(-2,2,2)--(0,2,2);
        \draw[thick](0,4,0)--(-2,4,0)--(-2,4,2)--(0,4,2);
        \draw[thick](0,4,-2)--(-2,4,-2)--(-2,4,0)--(0,4,0);
        \draw[dashed](0,2,-2)--(-2,2,-2)--(-2,2,0);
        \draw[dashed](0,0,0)--(-2,0,0)--(-2,0,2)--(0,0,2);
        \draw[thick](-2,0,2)--(0,0,2);
        \draw[thick](0,2,2)--(-2,2,2);
        \draw[dashed](-2,0,0)--(-2,4,0);
        \draw[thick](-2,0,2)--(-2,4,2);
        \draw[dashed](-2,2,-2)--(-2,4,-2);
        \draw(1,2,0) node{5};
        \draw(0,1,0) node{3};
        \draw(0,2,1) node{1};
        \draw(0,2,-1) node{2};
        \draw(0,3,0) node{4};
        \draw(-1,2,0) node{6};
    \end{tikzpicture}\qquad
    \begin{tikzpicture}[scale=.8]
        \node[circle, draw, fill=black, scale=.5] at (0,2,0) {};
        \draw[thick](2,2,0)--(0,2,0)--(0,2,2)--(2,2,2)--(2,2,0)--(2,0,0)--(2,0,2)--(0,0,2)--(0,2,2);
        \draw[thick](2,2,2)--(2,0,2);
        \draw[dashed](2,0,0)--(0,0,0)--(0,2,0);
        \draw[dashed](0,0,0)--(0,0,2);
        \draw[thick](0,2,0)--(0,4,0)--(0,4,2)--(0,2,2);
        \draw[thick](0,4,0)--(0,4,-2)--(0,2,-2)--(0,2,0);
        \draw[dashed](0,2,0)--(-2,2,0)--(-2,2,2)--(0,2,2);
        \draw[thick](0,4,0)--(-2,4,0)--(-2,4,2)--(0,4,2);
        \draw[thick](0,4,-2)--(-2,4,-2)--(-2,4,0)--(0,4,0);
        \draw[dashed](0,2,-2)--(-2,2,-2)--(-2,2,0);
        \draw[dashed](0,0,0)--(-2,0,0)--(-2,0,2)--(0,0,2);
        \draw[thick](-2,0,2)--(0,0,2);
        \draw[thick](0,2,2)--(-2,2,2);
        \draw[dashed](-2,0,0)--(-2,4,0);
        \draw[thick](-2,0,2)--(-2,4,2);
        \draw[dashed](-2,2,-2)--(-2,4,-2);
        \draw[thick](2,2,0)--(2,2,-2)--(0,2,-2);
        \draw(1,2,0) node{5};
        \draw(0,1,0) node{3};
        \draw(0,2,1) node{1};
        \draw(0,2,-1) node{2};
        \draw(0,3,0) node{4};
        \draw(-1,2,0) node{6};
    \end{tikzpicture}
    \caption{The connected component containing the fan triangulation of nonagon has the middle two $2$-dimensional and $3$-dimensional cubes}
    \label{fig:conj5nonagon}
\end{figure}
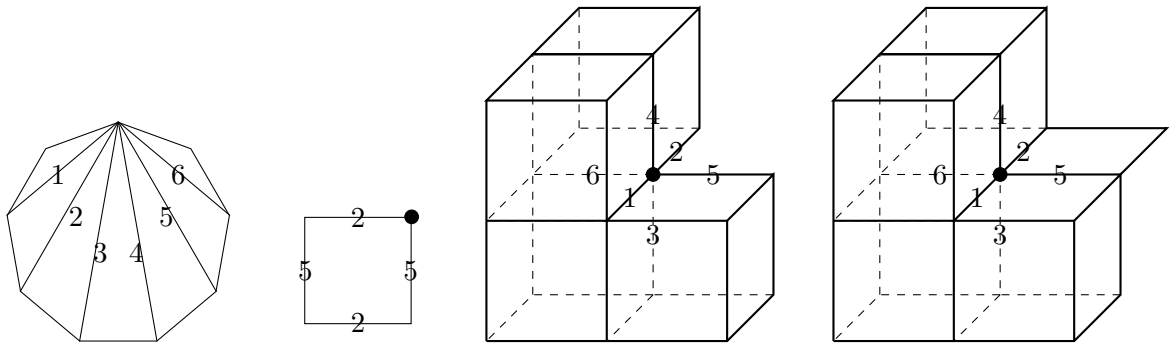

%%%%%%%%%%%%%%%
%
\section{Observations and a conjecture}\label{sec:observations}

\begin{figure}[h!]
\includegraphics[width = 0.8\textwidth]{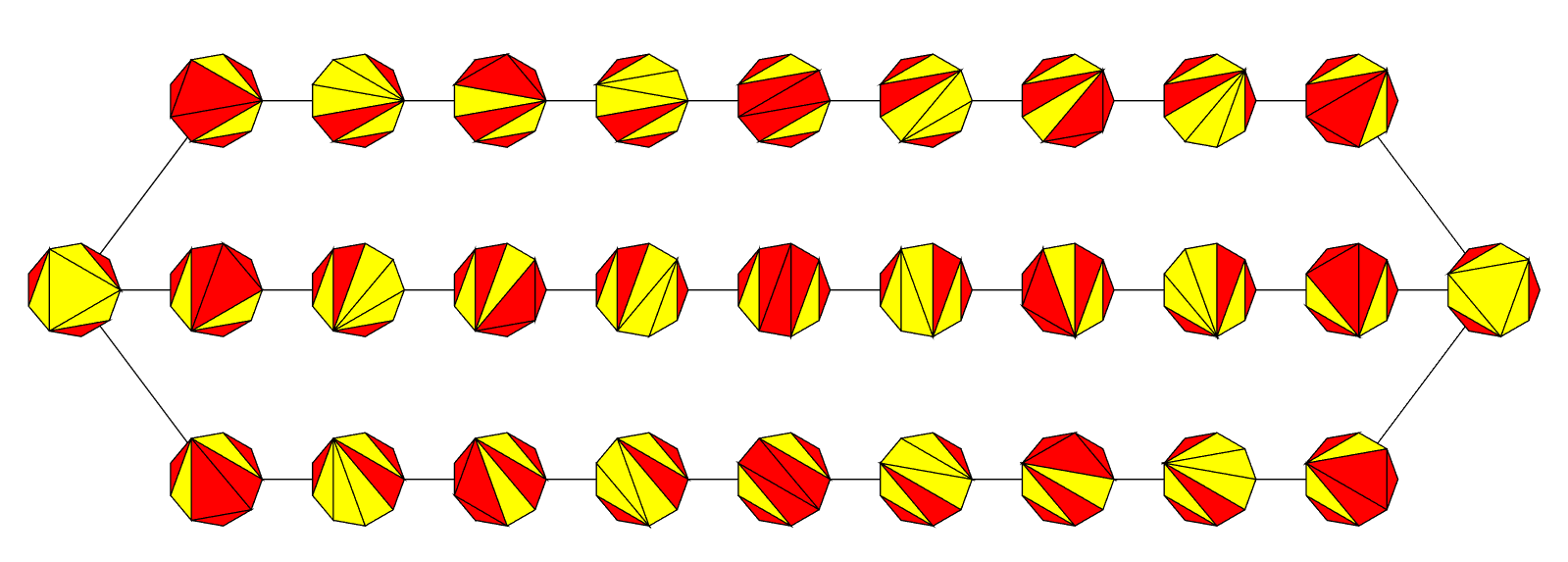}
    \caption{A connected component of the flip graph for coloured nonagon triangulations which contains no fans, and has a minimum cycle size 20.}
\label{fig:counterexampleforgraphstructure}
\end{figure}

We conclude this paper by a number of observations and a conjecture. Let $P_{n+2}$ be a convex $n+2$-gon. 
\begin{ob}
    For $n\le 4$ any  connected component of the 2-coloured flip graphs of $P_{n+2}$ is either a tree, or obtained from adding leaves onto a 4-cycle. 
    See \cref{app:comps}. 
    For $n>4$, this is not true anymore.  An example is a component for $n=7$ in \cref{fig:counterexampleforgraphstructure}. 
\end{ob}

\begin{ob}
    There are connected components of the flip graph which do not have any leaves, see e.g.\cref{fig:counterexampleforgraphstructure} for $n=7$ or~\ref{app:comps} in the case of $n=6$. 
\end{ob}

%\begin{ob}
%     Let $n\le 4$. Any two leaves of a connected component correspond to triangulations which are isometric to each other (i.e. can be transformed into each other using a rotation, a reflection or a combination of both). 
%\end{ob}

\begin{figure}[h!]
     \includegraphics[width = 0.5\textwidth]{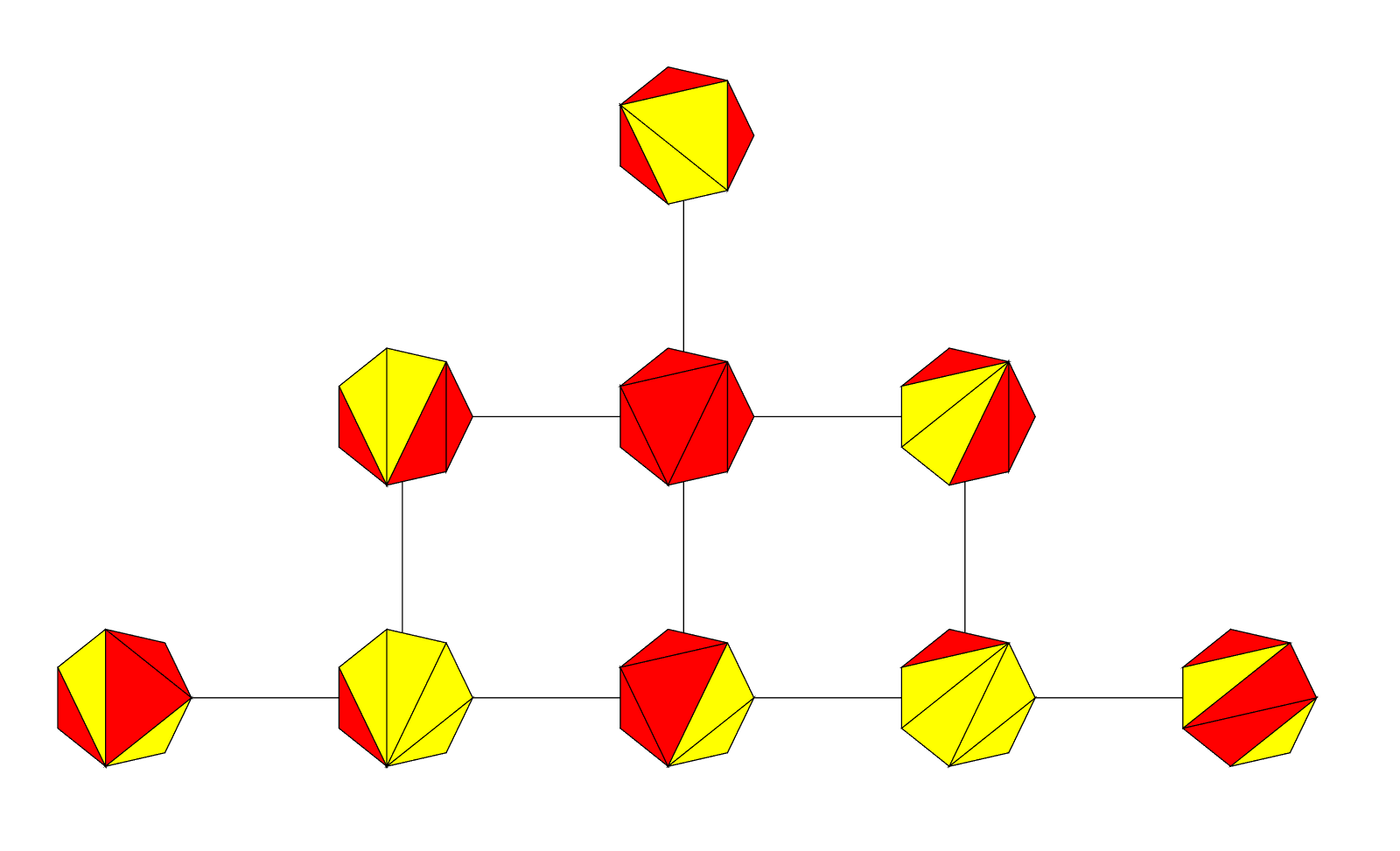}
     \caption{A section of a connected component for the heptagon, which has 3 leaves of 2 different triangulation types.}
     \label{fig:counterexampleleaves}
 \end{figure}

In the examples we considered, no two triangulations in a connected component contained two triangles with the same vertices but with different colours. See for example Figure~\ref{fig:counterexampleforgraphstructure} for an illustration. We suspect that this could be true in general: 

%this is true for at least hexagons
\begin{conj}
 In a connected component of the 2-coloured flip graph, a triangle cannot appear in the same position but with different colours. 
\end{conj}

%%%%%%%%%%%% appendix %%%%%%%%%%%%%%%

\medskip

\appendix

%%%5%%%%%
%
\section{Proof of Theorem \ref{thm1}}\label{sec:4colour-signed} 

We recall the statement of Theorem~\ref{thm1} from the Introduction: The Four-Colour Theorem holds if and only if for any two triangulations of a convex polygon, one can 2-colour them in a way that there exists a sequence of 2-coloured flips linking the two triangulations. 

This result by Gravier and Payan motivates the notion of coloured mutation. The work of Gravier and Payan has appeared in French in 2002. For the convenience of the reader, we summarize their reasoning in this section. 

%The following section is the summary equivalence between \cref{conj1} and \cref{thm:fourcolourthm} from the French article \cite{GRAVIER2002817} of Gravier and Payan which we stated as Theorem~\ref{thm1} in the Introduction. 

We first recall the notions needed. 
In this section, we will use `signed triangulations' to refer to 2-coloured triangulations in order to distinguishing from the notion of a colour in the 4-colour theorem. 

\begin{defn}
Let $P$ be a convex polygon. We introduce the following definitions:
    \begin{itemize}
        \item Let $T$ be a triangulation of $P$. We write $\mathcal{D}(T)$ for the set of all diagonals of $T$ and $\mathcal F(T)$ for its faces (the triangles). 
        \item A \emph{sub-polygon} $S \subset P$ is a polygon whose vertices are a subset of those of $P$, and which respects the  cyclic order of the vertices of $P$.
        \item $S-x$ denotes the sub-polygon induced by all vertices except $x$.
        \item A \emph{signed triangulation} of $P$ is a 2-coloured triangulation $T$ of the polygon, i.e. a pair $\{T,s\}$, where $s:\mathcal{F}(T) \rightarrow \{+1,-1\}$ is a 2-colouring of the triangles %((inner faces)) 
        of $T$. Let $\overline{s}$ be the signed triangulation obtained from $s$ by changing all signs. We write $(T,s)$ to denote the class $\{\{T,s\}, \{T,\overline{s}\}\}$.
        \item A {\em signed flip} is a 2-coloured flip of a diagonal of a signed triangulation. 
        \item 
        We recall the definition of a weighting of $P$ (Definition~\ref{def:weighting}) and introduce a notation suitable with the other terms of this section: 
        The pair $\{T,p\}$ where $p:V(T)\to \{-1,0,+1\}$ is a function on the vertices of $T$ (or of $P$) is called a {\em weighting} of $T$ is there exists a 2-colouring $s$ of $T$ such that for every vertex $x$ of $T$ we have $p(x)=\sum_{t\in \mathcal{F}(T)}s(t) \mod 3$. 
        Similarly as before, if $\{T,p\}$ is a weighting and $s$ a 2-colouring giving rise to it, we write $\{T,\overline{p}\}$ for the weighting associated to $\overline{s}$. We use $(T,p)$ to denote the valuation $p$ of $T$, up to exchanging $s$ with $\overline{s}$. 
        \item A \emph{valuation} of $T$ is a pair $(T,v)$, where $v:\mathcal{D}(T) \rightarrow \{0,1\}$ assigns $0$ or $1$ to every diagonal of $T$. 
        \item A \emph{colouring} of $T$ is a pair $(T,col)$, where $col$ is a 4-proper colouring of the vertices of $T$ (i.e., no two vertices adjacent under $T$  share the same colour). We will often use the letters $a,b,c,d$ to indicate the four colours of a colouring. We only consider colourings up to permutation of colours.
    \end{itemize}
\end{defn}

%A signed flip is only allowed if the adjacent triangles have the same sign. 
Let $\{T,s\}$ be a signed triangulation. The signs determine a weighting of $T$ by definition. There is a natural way to associate a valuation 
$(T,v)$ to any signed triangulation $\{T,s\}$ if a diagonal is incident with two triangles of the same sign, its valuation is set to be $0$. Otherwise, its valuation is set to be $1$. By definition, this procedure associates the same valuation $v$ to $\{T,\overline{s}\}$. So we can naturally assign a valuation $(T,v)$ to $(T,s)$.

\begin{exmp}
    See \cref{fig:signedtriangulation} for an example of a signed triangulation $T$ of heptagon, with associated weighting (on the left), valuation (in the middle) and with a colouring for $T$ (on the right).
    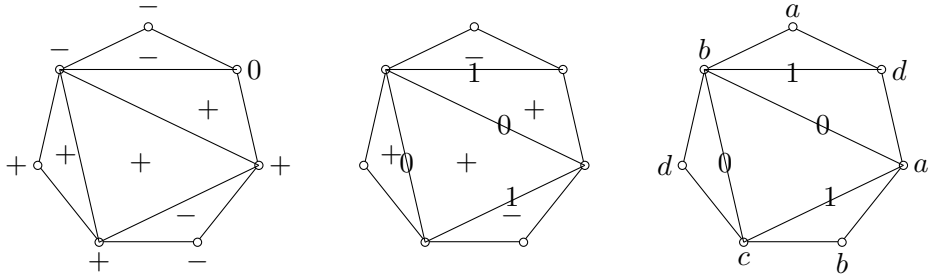
\begin{figure}[h!]
        \centering
        \begin{tikzpicture}[baseline]
            \node[draw, regular polygon, regular polygon sides=7, minimum size=3cm] (a) {};
            \foreach \i in {1,2,...,7}
                \node[circle, draw, fill=white, fill opacity=1, scale=0.3] at (a.corner \i) {};
            \node[above] at (a.corner 1) {$-$};
            \node[above] at (a.corner 2) {$-$};
            \node[left] at (a.corner 3) {$+$};
            \node[below] at (a.corner 4) {$+$};
            \node[below] at (a.corner 5) {$-$};
            \node[right] at (a.corner 6) {$+$};
            \node[right] at (a.corner 7) {$0$};
            \draw (a.corner 2) edge (a.corner 7);
            \draw (a.corner 2) edge (a.corner 4);
            \draw (a.corner 4) edge (a.corner 6);
            \draw (a.corner 6) edge (a.corner 2);
            \node at (-.1,-.3) {$+$};
            \node at (.8,.4) {$+$};
            \node at (0,1.1) {$-$};
            \node at (-1.1,-.2) {$+$};
            \node at (.5,-1) {$-$};
        \end{tikzpicture}\qquad
        \begin{tikzpicture}[baseline]
            \node[draw, regular polygon, regular polygon sides=7, minimum size=3cm] (a) {};
            \foreach \i in {1,2,...,7}
                \node[circle, draw, fill=white, fill opacity=1, scale=0.3] at (a.corner \i) {};
            \draw (a.corner 2) edge (a.corner 7);
            \draw (a.corner 2) edge (a.corner 4);
            \draw (a.corner 4) edge (a.corner 6);
            \draw (a.corner 6) edge (a.corner 2);
            \node at (-.1,-.3) {$+$};
            \node at (.8,.4) {$+$};
            \node at (0,1.1) {$-$};
            \node at (-1.1,-.2) {$+$};
            \node at (.5,-1) {$-$};
            \node at (0,.9) {$1$};
            \node at (-.9,-.3) {$0$};
            \node at (.4,.2) {$0$};
            \node at (.5,-.75) {$1$};
        \end{tikzpicture}\qquad
        \begin{tikzpicture}[baseline]
            \node[draw, regular polygon, regular polygon sides=7, minimum size=3cm] (a) {};
            \foreach \i in {1,2,...,7}
                \node[circle, draw, fill=white, fill opacity=1, scale=0.3] at (a.corner \i) {};
            \draw (a.corner 2) edge (a.corner 7);
            \draw (a.corner 2) edge (a.corner 4);
            \draw (a.corner 4) edge (a.corner 6);
            \draw (a.corner 6) edge (a.corner 2);
            \node[above] at (a.corner 1) {$a$};
            \node[above] at (a.corner 2) {$b$};
            \node[left] at (a.corner 3) {$d$};
            \node[below] at (a.corner 4) {$c$};
            \node[below] at (a.corner 5) {$b$};
            \node[right] at (a.corner 6) {$a$};
            \node[right] at (a.corner 7) {$d$};
            \node at (0,.9) {$1$};
            \node at (-.9,-.3) {$0$};
            \node at (.4,.2) {$0$};
            \node at (.5,-.75) {$1$};
        \end{tikzpicture}
        \caption{A signed triangulation of heptagon with weighting, valuation, colouring.}        \label{fig:signedtriangulation}
    \end{figure}
\end{exmp}

Notice that 
``signed triangulations,  weighting, valuation and colouring'' are equivalent notions, up to taking the opposite signs/weights:

\begin{enumerate}
    \item 
    $(T,s) \equiv (T,p)$: Weightings and 2-colourings are equivalent by definition. 
    %By the definition of $(T,p)$, one and only one weighting $p$ can be matched to any signed triangulation. 
    \item $(T,s) \equiv (T,v)$: 
    Any signed triangulation $(T,s)$ gives a valuation $(T,v)$ as we have explained above (for any diagonal $xy\in T$, $v(xy)=0$ if and only if the two triangles adjacent to $xy$ have the same sign). Conversely, any valuation $(T,v)$ gives rise to two signed triangulations $\{T,s\}$ and $\{T,\overline{s}\}$. 
    \item $(T,v) \equiv (T,col)$. 
    Given a valuation $(T,v)$, we construct a 4-colouring $col$ of the vertices of $P$ compatible with $T$, denoted by $col(T,v)$: Choose a vertex of degree 2 in $T$. Such a vertex lies in a triangle which has two boundary edges (every triangulation has at least two such triangles). 
    We colour the three vertices of this triangles in three different colours. We proceed as follows: for any quadrilateral with vertices $xyzt$, formed by two adjacent triangles sharing the common diagonal $yt$, we colour $x,z$ in the same colour if and only if the diagonal $yt$ is valued 1 under $v$. 
    Starting with the above triangle, we thus obtain a colouring of $T$ with (up to) four colours. The colouring $col(T,v)$ is unique up to permutation of the colours. 
    
    Reciprocally, starting from $(T,col)$, we get a valuation of $T$ by setting a diagonal of any quadrilateral to be 0 if and only if the four vertices of the quadrilateral this diagonal determines are all coloured differently. 
\end{enumerate}

By the above, it makes sense to write $(T,\varepsilon)$ where $\varepsilon$ is in $\{s,p,v,col\}$  as these are all equivalent.

\begin{remark}\label{rem:effect-flip} 
Let $T$ be a triangulation of a polygon. We comment on the effect of a flip on the notions weighting, valuation and colouring. See \cref{fig:signedflip} for an illustration. 
    
\begin{itemize}  
\item Any flippable diagonal has valuation $0$. If we flip it, the new diagonal also has valuation $0$ while the diagonals bounding the corresponding quadrilateral change their valuation. 
\item    
The weighting of the vertices remains unchanged under flips. 
\item     
Any colouring for $T$ is still a colouring for the new triangulation. 
\end{itemize}
\end{remark}

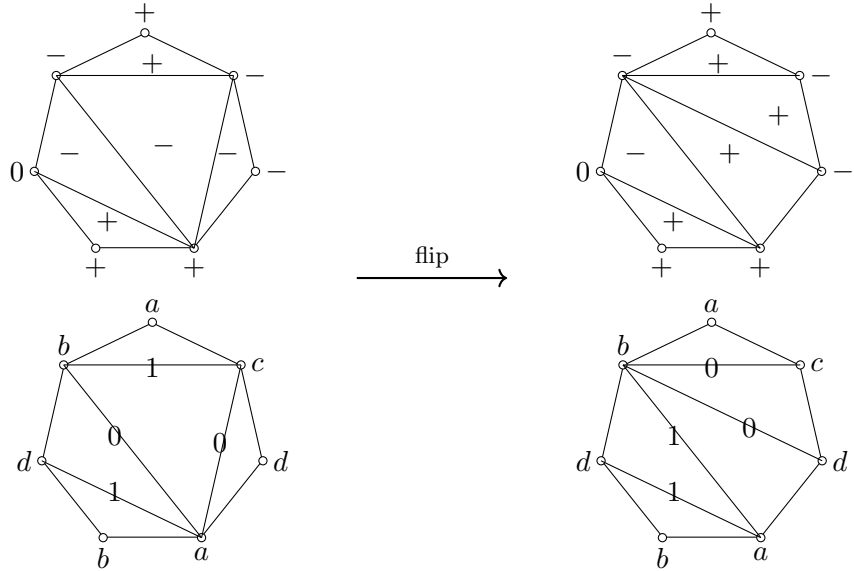
\begin{figure}[h!]
    \centering
    \begin{tikzpicture}
        \node[draw, regular polygon, regular polygon sides=7, minimum size=3cm] (a) {};
        \foreach \i in {1,2,...,7}
            \node[circle, draw, fill=white, fill opacity=1, scale=0.3] at (a.corner \i) {};
        \draw (a.corner 2) edge (a.corner 7);
        \draw (a.corner 2) edge (a.corner 5);
        \draw (a.corner 3) edge (a.corner 5);
        \draw (a.corner 5) edge (a.corner 7);
        \node[above] at (a.corner 1) {$+$};
        \node[above] at (a.corner 2) {$-$};
        \node[left] at (a.corner 3) {$0$};
        \node[below] at (a.corner 4) {$+$};
        \node[below] at (a.corner 5) {$+$};
        \node[right] at (a.corner 6) {$-$};
        \node[right] at (a.corner 7) {$-$};
        \node at (.25,0) {$-$};
        \node at (1.1,-.1) {$-$};
        \node at (.1,1.1) {$+$};
        \node at (-1,-.1) {$-$};
        \node at (-.5,-1) {$+$};
    \end{tikzpicture}\qquad
    \begin{tikzpicture}
        \draw [->, thick] (2,6) -- (4,6) node [pos=0.5,above,font=\footnotesize] {flip};
    \end{tikzpicture}\qquad
    \begin{tikzpicture}
        \node[draw, regular polygon, regular polygon sides=7, minimum size=3cm] (a) {};
        \foreach \i in {1,2,...,7}
            \node[circle, draw, fill=white, fill opacity=1, scale=0.3] at (a.corner \i) {};
        \draw (a.corner 2) edge (a.corner 7);
        \draw (a.corner 2) edge (a.corner 5);
        \draw (a.corner 3) edge (a.corner 5);
        \draw (a.corner 2) edge (a.corner 6);
        \node[above] at (a.corner 1) {$+$};
        \node[above] at (a.corner 2) {$-$};
        \node[left] at (a.corner 3) {$0$};
        \node[below] at (a.corner 4) {$+$};
        \node[below] at (a.corner 5) {$+$};
        \node[right] at (a.corner 6) {$-$};
        \node[right] at (a.corner 7) {$-$};
        \node at (.25,-.1) {$+$};
        \node at (.9,.4) {$+$};
        \node at (.1,1.1) {$+$};
        \node at (-1,-.1) {$-$};
        \node at (-.5,-1) {$+$};
    \end{tikzpicture}\\
    \begin{tikzpicture}
        \node[draw, regular polygon, regular polygon sides=7, minimum size=3cm] (a) {};
        \foreach \i in {1,2,...,7}
            \node[circle, draw, fill=white, fill opacity=1, scale=0.3] at (a.corner \i) {};
        \draw (a.corner 2) edge (a.corner 7);
        \draw (a.corner 2) edge (a.corner 5);
        \draw (a.corner 3) edge (a.corner 5);
        \draw (a.corner 5) edge (a.corner 7);
        \node[above] at (a.corner 1) {$a$};
        \node[above] at (a.corner 2) {$b$};
        \node[left] at (a.corner 3) {$d$};
        \node[below] at (a.corner 4) {$b$};
        \node[below] at (a.corner 5) {$a$};
        \node[right] at (a.corner 6) {$d$};
        \node[right] at (a.corner 7) {$c$};
        \node at (0,.9) {$1$};
        \node at (-.5,0) {$0$};
        \node at (.9,-.1) {$0$};
        \node at (-.5,-.75) {$1$};
    \end{tikzpicture}\qquad
    \begin{tikzpicture}
        \draw [->, thick, opacity=0] (2,6) -- (4,6) node [pos=0.5,above,font=\footnotesize] {flip};
    \end{tikzpicture}\qquad
    \begin{tikzpicture}
        \node[draw, regular polygon, regular polygon sides=7, minimum size=3cm] (a) {};
        \foreach \i in {1,2,...,7}
            \node[circle, draw, fill=white, fill opacity=1, scale=0.3] at (a.corner \i) {};
        \draw (a.corner 2) edge (a.corner 7);
        \draw (a.corner 2) edge (a.corner 5);
        \draw (a.corner 3) edge (a.corner 5);
        \draw (a.corner 2) edge (a.corner 6);
        \node[above] at (a.corner 1) {$a$};
        \node[above] at (a.corner 2) {$b$};
        \node[left] at (a.corner 3) {$d$};
        \node[below] at (a.corner 4) {$b$};
        \node[below] at (a.corner 5) {$a$};
        \node[right] at (a.corner 6) {$d$};
        \node[right] at (a.corner 7) {$c$};
        \node at (0,.9) {$0$};
        \node at (-.5,0) {$1$};
        \node at (.5,.1) {$0$};
        \node at (-.5,-.75) {$1$};
    \end{tikzpicture}
    \caption{The effect of a signed flip on  weighting, valuation, colouring.}
    \label{fig:signedflip}
\end{figure}

Note that a 3-colour colouring of a triangulation corresponds to the case where each diagonal has value 1, and such signed triangulation is called \emph{alternating}. 
Alternating signed triangulations are isolated vertices in the flip exchange graph and so they are not of interest for us.  

\begin{defn}
    Let $(T,\varepsilon)$ and 
    $(T', \varepsilon')$ be two signed triangulations of the same polygon. We write $(T,\varepsilon) \sim (T',\varepsilon')$ if there exists a sequence of 2-coloured flips from $(T,\varepsilon)$ to $(T',\varepsilon')$. 
    This sequence may be empty (i.e. we allow $T=T'$ with $\varepsilon=\varepsilon'$). 
    One can check that $\sim$
    is an equivalence relation, we denote the class of $(T,\varepsilon)$ by $[T,\varepsilon]$. 
\end{defn}

Now we are ready to prove~\cref{thm1} which we reformulate as follows: 

\begin{thm}\label{thm:equiv} 
Let $(T,v)\neq (T',v')$ be signed triangulations of $P$. 
%
%    Let $(T,v)$ and $(T',v')$ 
%    two different coloured triangulations of a polygon. 
Then $(T,v) \sim (T',v')$ if and only if 
$col(T,v)=col(T',v')$ and it uses 4 colours.
\end{thm}

\begin{proof}[Proof of $\Longrightarrow$ of Theorem~\ref{thm:equiv}]
Using Remark~\ref{rem:effect-flip} one can see that a coloured flip does not change the colouring of the vertices. Iterating, we get that  
$(T,v)\sim (T',v')$ implies 
$col(T,v)=col(T',v')$. 
Since we assumed that the two triangulations are different, the sequence of signed flips needed to go from $(T,v)$ to $(T',v')$ is not empty, i.e. the flip graph is not a single point and there is at least one diagonal valued with $0$. Hence $col(T,v)$ uses four colours. 
\end{proof}

To prove the converse of the theorem, we first show three lemmas. We have to study the vertices of $P$ and their neighbours. In a triangulated polygon any vertex of $P$ has neighbours on the boundary and potentially neighbours through diagonals of the triangulation. When dealing with the former, we refer to them as neighbours along the boundary (or on the boundary). 

\begin{lem}\label{lem:samecolour}
    Let $(T,\varepsilon)$ be a signed triangulation of a polygon $P$ and $x$ a vertex of $P$. Assume that the two neighbours of $x$ along the boundary are the only two neighbours of $x$ with the same colour. 
    Then $x$ has 3 or 4 neighbours, and $p(x)=0$. 
\end{lem}

\begin{proof}
    Clearly, $x$ cannot have only 2 neighbours in this case as in that case, these would belong to a common triangle with $x$. 
    
    Suppose for contradiction that the vertex $x$ has at least five neighbours. Then the two neighbours on the polygon are not the only two neighbours of the same colour in $T$: We can only colour three neighbours of $x$ with distinct colours (different from the colour of $x$). And we would have at least three vertices of the same colour or another pair of neighbours with the same colour. Hence $x$ cannot have more than 4 neighbours.

    In case $x$ has three neighbours, these four vertices span a quadrilateral (with $x$) and the diagonal ending at $x$ has value $1$ as the other end must be of a different colour. In particular, the two triangles incident with $x$ have opposite sign and $x$ has weight $0$. 

    In case $x$ has four neighbours, the two neighbours which are linked to $x$ by diagonals must be of two different colours which are also different from the colour of $x$. In particular, both these diagonals have value $0$. Therefore, $x$ is incident with three triangles of the same sign and the weight $p(x)$ is $0$ (mod $3$). 
    % the diagonal of the triangulation of the sub-polygon induced by $x$ and its neighbours is valued 1. So $x$ has zero weight. When $x$ has 4 neighbours, the rest two neighbours are coloured in two colours different from $x$ and its neighbours on the polygon. So the two diagonals of the triangulation of the sub-polygon induced by $x$ and its neighbours are valued 0. Hence, $x$ has zero weight. 
\end{proof}

\begin{lem}\label{lem:notsamecolour}
    Let $(T,\varepsilon)$ be a signed triangulation of $P$ and $x$ a vertex of $P$. If $x$ has no two neighbours of the same colour, then $x$ has 2 or 3 neighbours and the weight $p(x)$ of $x$ is not $0$. %and the weight of $x$ is not zero. 
\end{lem}

\begin{proof}
    It is clear that $x$ can only have 2 or 3 neighbours as if there are more, there would be at least two of them with the same colour. %Suppose for contradiction that Suppose that $(T,\varepsilon)$ is such a counterexample where the vertex $x$ has at least 4 neighbours. Since none of its neighbours are in the same colour, we must colour them in at least 5 colours. This is a contradiction.

    In case $x$ has only two neighours, it is incdent with only one triangle and so $p(x)$ is $1$ or $2$ (mod $3$).  
    
    So assume that $x$ has three neighbours. In the quadrilateral spanned by $x$ and its three neighbours, $T$ has a diagonal connecting $x$ with the fourth vertex, say $y$. The vertices $x$ and $y$ have to be of different colour and so all four vertices of this quadrilateral are of different colours. Hence the diagonal $xy$ has value $0$. So the two triangles at $x$ are of the same sign and the weight $p(x)$ is in $\{1,2\}$ $\mod 3$. 
%    case $x$ has 3 neighbours, $x$ is also not of zero weight since the diagonal of the triangulation of the sub-polygon induced by $x$ and its neighbours is valued 0.
\end{proof}    

% \begin{notation}
% If $T$ is a triangulation of a polygon $P$ and $x$ a vertex of $P$, we write $G(x)$ for the subpolygon induced by $x$: it is the full subgraph of $T$ on the vertices connected to $x$ through $T$ (including $x$). 
% \end{notation}

\begin{lem}\label{lem:zeroweight}
Let $(T, \varepsilon)$ be a signed triangulation of a polygon $P$. Let $x$ be a vertex of $P$. Then $p(x)=0$ if and only if its two neighbours on $P$ have the same colour.
\end{lem}

\begin{proof} 
    It is enough to consider the full subgraph of the triangulated polygon induced by $x$ (it consists of $x$, of all vertices connected with $x$ and of all boundary edges and diagonals connecting them). 
    The idea is to use induction on the degree of the vertex $x$. 

    (1) If $x$ has no two neighbours of the same colour, then $x$ has degree $2$ or $3$ and  $p(x)\ne 0$ by Lemma~\ref{lem:notsamecolour}. 

    (2) If the two neighbours of $x$ on the polygon are the only neighbours of $x$ with the same colour, then $x$ has degree $3$ or $4$ and $p(x)=0$ by Lemma~\ref{lem:samecolour}. 
    With (1) and (2) we have covered all cases where $x$ has degree $2$ or $3$ (in degree $3$, if there are vertices of the same colour among the neighbours of $x$, they have to be on the boundary, for a colouring to be valid). 

    So the result holds for vertices $x$ of degree $\le 3$. 
    
    (3) It remains to check the general situation. Let $y_1,y_2,\dots, y_r$ be the neighbours of $x$, with $y_1$ and $y_r$ being along the boundary and where $r\ge 4$. See left hand picture of Figure~\ref{fig:colour-neighbours}

    Since $r\ge 4$, there are vertices among the $y_i$ of the same colour. Pick $y_i$, $y_j$, $i<j-1$ of the same colour such that there are no two vertices of the same colour among $y_{i+1},\dots, y_{j-1}$. Consider the triangulated subpolygon on the vertices $x,y_i,y_{i+1},\dots, y_j$. Using the same argument as in Lemma~\ref{lem:samecolour}, we find that either $j=i+2$ or $j=i+3$ and that the triangles incident with $x$ and that $p(x)=0$ in this subpolygon (there are either two triangles of opposite signs or three triangles of the same sign). 

    We then identify $y_i$ with $y_j$, getting a new polygon $P'$, reducing the degree of $x$ in it, see right hand side of Figure~\ref{fig:colour-neighbours}. So in the polygon $P'$, the weight of $x$ is $0$ if and only if the two neigbhours $y_1$ and $y_r$ on the boundary have the same colour. 
    Since the region between $y_i$ and $y_j$ contributes by $0$ to the weight, the claim holds. 
%    Notice that it suffices to prove this lemma on the graph induced by the vertex $x$ and its neighbours. Then we prove by induction on the degree of $x$. We observe that if the two neighbours of $x$ on the polygon $P$ are the only neighbours of $x$ of the same colour, then $x$ has 3 or 4 neighbours. Thus, $x$ has zero weight by \cref{lem:samecolour}. If there are not two neighbours of $x$ of the same colour, then $x$ has 2 or 3 neighbours. Thus, $x$ has non-zero weight by \cref{lem:notsamecolour}.
    
%    Now it remains to check the case where $x$ has two neighbours of the same colour, but they are not both neighbours on the polygon. Assume $x$ has two neighbours $a,b$ of the same colour, and at least one of them is not a neighbour of $x$ on the polygon. By induction, on the triangulated sub-polygon induced by $x,a,b$, and all other neighbours of $x$ on the arc $ab$, the weight of $x$ is zero. Now we identify $a$ and $b$. By induction, the weight of $x$ in the polygon is thus zero if and only if its two neighbours on the polygon (which are identical to those on $P$) have the same colour.

%    Hence the weight of $x$ on $P$ is zero if and only if the two neighbours of $x$ on $P$ have the same colour.
\end{proof}

\begin{figure}
    \centering
    \includegraphics[scale=.8]{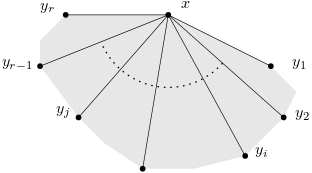} 
    \hskip 1.5cm \includegraphics[scale=.8]{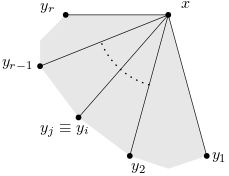} \caption{The neighbourhood of $x$ in $P$ and in $P'$}
    \label{fig:colour-neighbours}
\end{figure}

\begin{proof}[Proof of $\Longleftarrow$ of Theorem~\ref{thm:equiv}]
Assume that there exists a polygon $P$ and two triangulations $(T,v)$ and $(T'v')$ of $P$ which provide a counterexample. Let $P$ be minimal with this property. The polygon $P$ has at least $5$ vertices (one can check that the theorem is true for $4$ vertices). So $col(T,v)=col(T',v')$, this 
colouring uses all four colours, and there is no sequence of signed flips between these two signed triangulations. 
Among the vertices of $T$ of degree 2 we choose a vertex $x$ with the property that $T-x$ (the triangulated polygon without the triangle at $x$) is still coloured with four colours. Such a vertex always exists as $P$ has at least $5$ vertices and among them at least two vertices of degree $2$. At least one of them satisfies this condition (if one removes a degree 2 vertex $y$ and the remaining colouring only uses $3$ colours, one replaces $y$ by another degree $2$ vertex in $T$). 
Since all four colours are present in $T-x$, there exists a diagonal with valuation $0$. 

If there exists a signed triangulation $T''$ in the equivalence class $[T',v']$ where $x$ has degree $2$, then by minimality of the size of $P$, we know that for the polygon $P-x$ we have $(T-x,v)\sim (T''-x,v'')$. But then 
$(T,v)\sim (T'',v'')$ and the latter is in the equivalence class of $(T',v')$, so $(T,v)\sim (T',v')$, a contradiction. 

So we can assume that $x$ has degree $\ge 3$ in every triangulation in $[T',v']$. 

We partition this equivalence class into two sets $\mathcal T_1$ and $\mathcal T_2$. We will show that these are both empty, thus proving that no counter-example to the implication $\Leftarrow$ exists. 

We define $\mathcal T_1$ to be the set of all signed triangulations in $[T',v']$ having a diagonal of value $0$ incident with $x$. The set $\mathcal T_2$ are the ones where every diagonal at $x$ has value $1$. These are the signed triangulations which are alternating on the subpolygon induced by $x$ and all its neighbours in $T'$. 
(Since the degree of $x$ is at least $3$ for any signed triangulation in $[T',v']$, there is always at least one diagonal at $x$).

\noindent
\underline{Claim: $\mathcal T_1$ is empty:} \\ 
From the elements of $\mathcal{T}_1$ choose a signed triangulation $(T'',v'')$ where $x$ has minimal degree (this degree is $\ge 3$ as we have seen). The two neighbours of $x$ (along the boundary of the polygon) are adjacent in $T$ (as $x$ has degree $2$ in $T$) and so have different colour. By Lemma~\ref{lem:zeroweight}, this means that $p''(x)\ne 0$, where $p''$ is the weighting of $(T'',v'')$. This weighting is the same as that of $(T',v')$ and as that of $(T,v)$ as their colourings are the same. 
If there is a diagonal of value $1$ incident with $x$, say $xy_k$ (for some $k$), we flip a diagonal with value $0$ next to this diagonal. Then the diagonal $xy_k$ has value $0$. In this new triangulation, the degree of $x$ has gone down by one and we reach a contradiction. 
So all diagonals at $x$ must have value $0$. We flip the first such diagonal at $x$ (e.g. going clockwise through these diagonals). The result is a triangulation where either $x$ has degree $2$ (contradicting that the vertex $x$ has degree $>2$ for all elements of $[T',v']$) or it has degree $3$ and no diagonal of value $0$ incident with it, implying that $p''(x)=0$ (a contradiction to $p''(x)\ne 0$) or the resulting triangulation is an element of $\mathcal T_1$ where $x$ has smaller degree. 
Figure~\ref{fig:exampleofdegree3} illustrates the first two of these cases. 
In all three cases, this leads to a contradiction. 
Therefore, $\mathcal T_1$ is empty.

\noindent
\underline{Claim: $\mathcal T_2$ is empty:}\\
Recall that the signed triangulation of the subpolygon induced by $x$ and its neighbours in $T'$ is alternating (all diagonals at $x$ have value $1$). For any $(Q,\varepsilon)$ an element of $\mathcal T_2$, we write $P(Q)$ the maximal alternating subpolygon (maximal by inclusion) which contains $x$ and its neighbours. 
Let $(T'',v'')$ be an element of $\mathcal T_2$ which minimizes the size of $P(T'')$. 
As $P(T'')$ is maximal as alternating signed polygon, the boundary edges of $P(T'')$ which are diagonals in the original triangulated polygon have to have value $0$. Since  
$col(T'',v'')=col(T',v')$ and all four colours appear, there exists at least one diagonal of value $0$ (so such a boundary edge of $P(T'')$ has to exist). 
If we flip this diagonal, we obtain a new triangulation $S$. 
If this diagonal is incident with two edges (two diagonals or one diagonal and a boundary edge) at $x$, $S$ belongs to $\mathcal T_1$ (see Figure~\ref{fig:exampleofT_1} for an illustration). However, $\mathcal T_1$ is empty. 

Otherwise, $S$ belongs to $\mathcal T_2$ with $P(S)$ smaller than $P(T'')$ (see Figure~\ref{fig:exampleofT_2} for an illustration), also a contradiction. 
So $\mathcal T_2$ is empty. 
\end{proof}

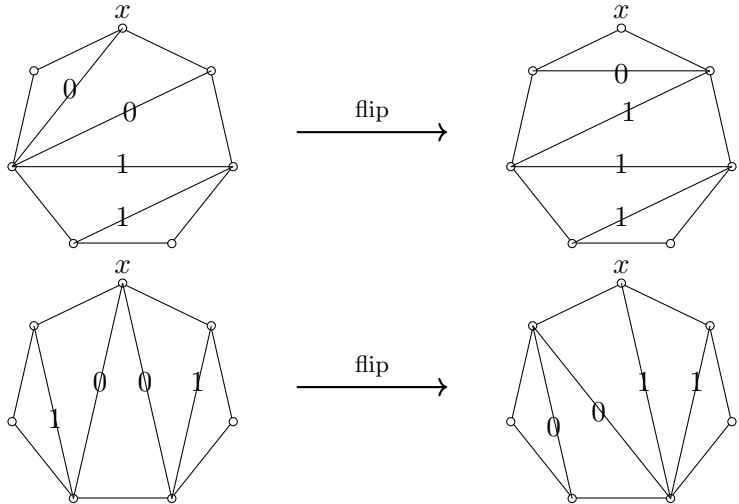
\begin{figure}[h!]
    \centering
    \begin{tikzpicture}[baseline]
        \node[draw, regular polygon, regular polygon sides=7, minimum size=3cm] (a) {};
        \foreach \i in {1,2,...,7}
            \node[circle, draw, fill=white, fill opacity=1, scale=0.3] at (a.corner \i) {};
        \node[above] at (a.corner 1) {$x$};
        \draw (a.corner 1) edge (a.corner 3);
        \draw (a.corner 3) edge (a.corner 7);
        \draw (a.corner 3) edge (a.corner 6);
        \draw (a.corner 4) edge (a.corner 6);
        \node at (-.7,.7) {0};
        \node at (.1,.4) {0};
        \node at (0,-.3) {1};
        \node at (0,-1) {1};
    \end{tikzpicture}\qquad
    \begin{tikzpicture}
        \draw [->, thick] (2,6) -- (4,6) node [pos=0.5,above,font=\footnotesize] {flip};
    \end{tikzpicture}\qquad
    \begin{tikzpicture}[baseline]
        \node[draw, regular polygon, regular polygon sides=7, minimum size=3cm] (a) {};
        \foreach \i in {1,2,...,7}
            \node[circle, draw, fill=white, fill opacity=1, scale=0.3] at (a.corner \i) {};
        \node[above] at (a.corner 1) {$x$};
        \draw (a.corner 2) edge (a.corner 7);
        \draw (a.corner 3) edge (a.corner 7);
        \draw (a.corner 3) edge (a.corner 6);
        \draw (a.corner 4) edge (a.corner 6);
        \node at (0,.9) {0};
        \node at (.1,.4) {1};
        \node at (0,-.3) {1};
        \node at (0,-1) {1};
    \end{tikzpicture}\\
    \begin{tikzpicture}[baseline]
        \node[draw, regular polygon, regular polygon sides=7, minimum size=3cm] (a) {};
        \foreach \i in {1,2,...,7}
            \node[circle, draw, fill=white, fill opacity=1, scale=0.3] at (a.corner \i) {};
        \node[above] at (a.corner 1) {$x$};
        \draw (a.corner 1) edge (a.corner 4);
        \draw (a.corner 1) edge (a.corner 5);
        \draw (a.corner 2) edge (a.corner 4);
        \draw (a.corner 5) edge (a.corner 7);
        \node at (1,.2) {1};
        \node at (.3,.2) {0};
        \node at (-.3,.2) {0};
        \node at (-.9,-.3) {1};
    \end{tikzpicture}\qquad
    \begin{tikzpicture}
        \draw [->, thick] (2,6) -- (4,6) node [pos=0.5,above,font=\footnotesize] {flip};
    \end{tikzpicture}\qquad
    \begin{tikzpicture}[baseline]
        \node[draw, regular polygon, regular polygon sides=7, minimum size=3cm] (a) {};
        \foreach \i in {1,2,...,7}
            \node[circle, draw, fill=white, fill opacity=1, scale=0.3] at (a.corner \i) {};
        \node[above] at (a.corner 1) {$x$};
        \draw (a.corner 2) edge (a.corner 5);
        \draw (a.corner 1) edge (a.corner 5);
        \draw (a.corner 2) edge (a.corner 4);
        \draw (a.corner 5) edge (a.corner 7);
        \node at (1,.2) {1};
        \node at (.3,.2) {1};
        \node at (-.3,-.2) {0};
        \node at (-.9,-.4) {0};
    \end{tikzpicture}
    \caption{Examples where $x$ becomes a vertex of degree $2$ respectively of degree $3$.}
    \label{fig:exampleofdegree3}
\end{figure}

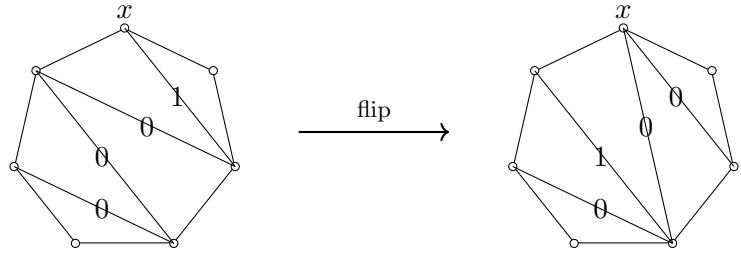
\begin{figure}[h!]
    \centering
    \begin{tikzpicture}[baseline]
        \node[draw, regular polygon, regular polygon sides=7, minimum size=3cm] (a) {};
        \foreach \i in {1,2,...,7}
            \node[circle, draw, fill=white, fill opacity=1, scale=0.3] at (a.corner \i) {};
        \node[above] at (a.corner 1) {$x$};
        \draw (a.corner 2) edge (a.corner 6);
        \draw (a.corner 1) edge (a.corner 6);
        \draw (a.corner 2) edge (a.corner 5);
        \draw (a.corner 3) edge (a.corner 5);
        \node at (.7,.6) {1};
        \node at (.3,.2) {0};
        \node at (-.3,-.2) {0};
        \node at (-.3,-.9) {0};
    \end{tikzpicture}\qquad
    \begin{tikzpicture}
        \draw [->, thick] (2,6) -- (4,6) node [pos=0.5,above,font=\footnotesize] {flip};
    \end{tikzpicture}\qquad
    \begin{tikzpicture}[baseline]
        \node[draw, regular polygon, regular polygon sides=7, minimum size=3cm] (a) {};
        \foreach \i in {1,2,...,7}
            \node[circle, draw, fill=white, fill opacity=1, scale=0.3] at (a.corner \i) {};
        \node[above] at (a.corner 1) {$x$};
        \draw (a.corner 1) edge (a.corner 5);
        \draw (a.corner 1) edge (a.corner 6);
        \draw (a.corner 2) edge (a.corner 5);
        \draw (a.corner 3) edge (a.corner 5);
        \node at (.7,.6) {0};
        \node at (.3,.2) {0};
        \node at (-.3,-.2) {1};
        \node at (-.3,-.9) {0};
    \end{tikzpicture}
    \caption{Example with $S \in \mathcal{T}_1$}
    \label{fig:exampleofT_1}
\end{figure}

\begin{figure}[h!]
    \centering
    \begin{tikzpicture}[baseline]
            \node[draw, regular polygon, regular polygon sides=7, minimum size=3cm] (a) {};
            \foreach \i in {1,2,...,7}
                \node[circle, draw, fill=white, fill opacity=1, scale=0.3] at (a.corner \i) {};
            \node[above] at (a.corner 1) {$x$};
            \draw[ultra thick] (a.corner 3) edge (a.corner 6);
            \draw[ultra thick] (a.corner 1) edge (a.corner 2);
            \draw[ultra thick] (a.corner 2) edge (a.corner 3);
            \draw[ultra thick] (a.corner 1) edge (a.corner 7);
            \draw[ultra thick] (a.corner 7) edge (a.corner 6);
            \draw (a.corner 1) edge (a.corner 6);
            \draw (a.corner 2) edge (a.corner 6);
            \draw (a.corner 3) edge (a.corner 6);
            \draw (a.corner 3) edge (a.corner 5);
            \node at (.7,.6) {1};
            \node at (.2,.3) {1};
            \node at (0,-.3) {0};
            \node at (-.3,-.9) {0};
        \end{tikzpicture}\qquad
        \begin{tikzpicture}
            \draw [->, thick] (2,6) -- (4,6) node [pos=0.5,above,font=\footnotesize] {flip};
        \end{tikzpicture}\qquad
        \begin{tikzpicture}[baseline]
            \node[draw, regular polygon, regular polygon sides=7, minimum size=3cm] (a) {};
            \foreach \i in {1,2,...,7}
                \node[circle, draw, fill=white, fill opacity=1, scale=0.3] at (a.corner \i) {};
            \node[above] at (a.corner 1) {$x$};
            \draw[ultra thick] (a.corner 2) edge (a.corner 6);
            \draw[ultra thick] (a.corner 1) edge (a.corner 2);
            \draw[ultra thick] (a.corner 1) edge (a.corner 7);
            \draw[ultra thick] (a.corner 7) edge (a.corner 6);
            \draw (a.corner 1) edge (a.corner 6);
            \draw (a.corner 2) edge (a.corner 6);
            \draw (a.corner 2) edge (a.corner 5);
            \draw (a.corner 3) edge (a.corner 5);
            \node at (.7,.6) {1};
            \node at (.2,.3) {0};
            \node at (-.2,-.3) {0};
            \node at (-.5,-.8) {1};
        \end{tikzpicture}
    \caption{Example with $S \in \mathcal{T}_2$}
    \label{fig:exampleofT_2}
\end{figure}
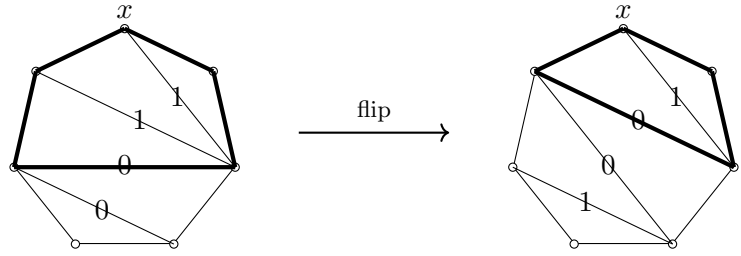

%%%%%%%%%%%%%%%
%
\section{Connected components of flip graphs}\label{app:comps}

In this appendix, we describe the connected components of the 2-coloured flip graphs of $P_{n+2}$ for $n\le 6$. We omit the isolated vertices. 
%include the general shapes of connected components of coloured flip graphs for triangulations of $P_{n+2}$ up to $n = 6$. We omit the isolated points since any coloured flip graph contains at least one. 
%The figures below were kindly provided by Karin Baur and Mark Parsons.

\begin{itemize}
    \item \emph{n = 2}. There is only one type of (non-trivial) connected components.
$$
\xymatrix@=2mm@R=-1mm{ 
\circ\ar@{-}[rr] &&\circ 
}
$$
    \item \emph{n = 3}. There is only one type of connected components.
$$
\xymatrix@=2mm@R=-1mm{ 
\circ\ar@{-}[rr] &&\circ \ar@{-}[rr]  && \circ
}
$$
    \item \emph{n = 4}. There are four different shapes of connected components.
$$
\xymatrix@=2mm@R=-1mm{ 
\circ\ar@{-}[rr] &&\circ \ar@{-}[rr] && \circ \ar@{-}[rr] && \circ \ar@{-}[rr] && \circ
} \qquad \xymatrix@=2mm@R=-1mm{ 
\circ\ar@{-}[rr] &&\circ \ar@{-}[rr] && \circ \ar@{-}[rr] && \circ
}
$$ 
$$
\xymatrix@=2mm@R=-1mm{ 
 &&& \circ\ar@{-}[rd] \\
\circ\ar@{-}[rr] && \circ\ar@{-}[dr]\ar@{-}[ru] && 
 \circ\ar@{-}[rr] && \circ\\
&&& \circ\ar@{-}[ru]
} \qquad \xymatrix@=2mm@R=1mm{ 
 & \circ\ar@{-}[d] \\
\circ\ar@{-}[r] & \circ \ar@{-}[r] & \circ 
% N\ar@{|->}[rr] & & \{e'\} \\
%G\ar[rr]^{\varphi}\ar[rd]_{\pi} && G' , &  N\ar@{->}[r]^{\varphi} &  \{e'\}\\ 
% & G/N\ar@{..>}[ru]_{\exists\,!\,\overline{\varphi}}
}
$$ 
\item \emph{n = 5}. 
The seven shapes of the different connected components are: 
$$
\xymatrix@=2mm@R=2mm{ 
\circ\ar@{-}[d] \\ 
\circ\ar@{-}[d] \ar@{-}[r] & \circ\ar@{-}[d] \\ 
\circ\ar@{-}[d] \ar@{-}[r] & \circ\ar@{-}[d]\ar@{-}[r] & \circ\ar@{-}[d] \\ 
\circ\ar@{-}[r] & \circ\ar@{-}[r] & \circ\ar@{-}[r] & \circ \ar@{-}[r] & \circ \ar@{-}[r] & \circ
} \qquad
\xymatrix@=2mm@R=2mm{ 
\circ\ar@{-}[d] \\ 
\circ\ar@{-}[d] \\ 
\circ\ar@{-}[d] \ar@{-}[r] & \circ\ar@{-}[d] \\ 
\circ\ar@{-}[d] \ar@{-}[r] & \circ\ar@{-}[d]\ar@{-}[r] & \circ\ar@{-}[d] \\ 
\circ\ar@{-}[r] & \circ\ar@{-}[r] & \circ\ar@{-}[r] & \circ \ar@{-}[r] & \circ
} \qquad \xymatrix@=2mm@R=2mm{ 
\circ\ar@{-}[d] \\ 
\circ\ar@{-}[d] \ar@{-}[r] & \circ\ar@{-}[d] \\ 
\circ\ar@{-}[d] \ar@{-}[r] & \circ\ar@{-}[d]\ar@{-}[r] & \circ\ar@{-}[d] \\ 
\circ\ar@{-}[r] & \circ\ar@{-}[r] & \circ\ar@{-}[r] &  \circ
}
$$ 
$$
\xymatrix@=4mm@R=-1mm{ 
\circ\ar@{-}[r] &\circ \ar@{-}[r] & \circ \ar@{-}[r] & \circ \ar@{-}[r] & \circ \ar@{-}[r] & \circ \ar@{-}[r] 
& \circ \ar@{-}[r] & \circ \ar@{-}[r] & \circ
} \qquad \xymatrix@=2mm@R=1mm{ 
\circ\ar@{-}[r] & \circ \ar@{-}[r] & \circ \ar@{-}[r] & \circ \ar@{-}[r] & \circ 
}
$$
$$
\xymatrix@=2mm@R=2mm{ 
\circ\ar@{-}[d] \\ 
\circ\ar@{-}[d] \ar@{-}[r] & \circ\ar@{-}[d] \\ 
\circ\ar@{-}[d] \ar@{-}[r] & \circ\ar@{-}[d]\ar@{-}[r] & \circ \\ 
\circ\ar@{-}[r]\ar@{-}[d] & \circ \\
\circ
} \qquad \xymatrix@=2mm@R=1mm{ 
\circ\ar@{-}[d] \\
\circ\ar@{-}[d]\ar@{-}[r] & \circ \ar@{-}[r] & \circ \ar@{-}[r] & \circ \\
\circ 
}
$$ 
\end{itemize}

\begin{itemize}
\item \emph{n = 6}. 
The 26 shapes of the different connected components are: 
$$
\xymatrix@=2mm@R=2mm{ 
 & & & & \circ\ar@{-}[ddd]\ar@/^2pc/@{-}[rrrrddd]   \\
 & \\ 
 & \circ\ar@{-}[rd] && \circ\ar@{-}[ld]\ar@{-}[rd] && \circ \ar@{-}[ld]\ar@{-}[rd]&& \circ\ar@{-}[ld] \\
  \circ\ar@/^2pc/@{-}[rrrruuu]\ar@{-}[rrrrdddd]\ar@{-}[dd]\ar@{-}[rrrd] && \circ\ar@{-}[rd] && 
   \circ\ar@{-}[ld]\ar@{-}[rd] && \circ\ar@{-}[ld] && \circ\ar@{-}[lllldddd]\ar@{-}[dd]\ar@{-}[llld] \\
&  &  & \circ\ar@{-}[rd]  & & \circ\ar@{-}[ld] \\ 
\circ\ar@{-}[rrrddd] &&&& \circ &&&& \circ\ar@{-}[lllddd] \\
 & \\
 &&&& \circ\ar@{-}[ld]\ar@{-}[rd] &&&& \\
 &&& \circ\ar@{-}[dr] && \circ\ar@{-}[ld] \\
 &&&& \circ\ar@{-}[d] \\
 &&&& \circ\ar@{-}[ld]\ar@{-}[rd] \\
 &&& \circ && \circ 
 } \qquad 
 \xymatrix@=2mm@R=2mm{ 
 &&&& \circ\ar@{-}[ld]\ar@{-}[rd] \\
   &&&  \circ\ar@{-}[ld]\ar@{-}[rd]\ar@/_6pc/@{-}[dddd] & & 
    \circ\ar@{-}[rd]\ar@{-}[ld]\ar@/^6pc/@{-}[dddd] &   \\
   && \circ\ar@{-}[ld]\ar@{-}[rd] & & \circ\ar@{-}[rd]\ar@{-}[ld]\ar@{-}[dd] &&\circ\ar@{-}[rd]\ar@{-}[ld]  & \\
  \circ & \circ\ar@{-}[l] && \circ & & \circ &&\circ \ar@{-}[r] &\circ  \\
 &&&& \circ\ar@{-}[ld]\ar@{-}[rd] &&&&\\
  &&& \circ\ar@{-}[ld]\ar@{-}[rd] && \circ \ar@{-}[ld]\ar@{-}[rd]& \\
 && \circ\ar@{-}[rd] && 
   \circ\ar@{-}[ld]\ar@{-}[rd] && \circ\ar@{-}[ld] \\
  &&  & \circ\ar@{-}[rd]  & & \circ\ar@{-}[ld] \\ 
&&&& \circ\ar@{-}[d] &&& \\
 &&&& \circ \\
 }
 $$
 
$$ \xymatrix@=2mm@R=2mm{ 
& & & & \circ\ar@{-}[d]  \\ 
\circ\ar@{-}[r] & \circ\ar@{-}[r] & \circ\ar@{-}[r] & \circ \ar@{-}[r] & \circ\ar@{-}[r] & \circ\ar@{-}[r] 
& \circ\ar@{-}[r] & \circ\ar@{-}[r] & \circ
} \qquad
\xymatrix@=2mm@R=1mm{ 
\circ\ar@{-}[r] & \circ \ar@{-}[r] & \circ \ar@{-}[r] 
 & \circ \ar@{-}[r] & \circ \ar@{-}[r] & \circ 
}
$$ 
$$
\xymatrix@=2mm@R=2mm{ 
& \circ \ar@{-}[d] & &  \circ\ar@{-}[d] \\ 
\circ\ar@{-}[d] \ar@{-}[r] & \circ\ar@{-}[d]\ar@{-}[r] & \circ\ar@{-}[d]\ar@{-}[r] & 
 \circ\ar@{-}[d] \\ 
\circ\ar@{-}[r]\ar@{-}[d] & \circ\ar@{-}[r] & \circ\ar@{-}[r]\ar@{-}[d] & \circ \\
\circ  &   & \circ
} \qquad 
\xymatrix@=2mm@R=2mm{ 
 & & &  \circ\ar@{-}[d] \\
 & & &  \circ\ar@{-}[rd]\ar@{-}[ld] \\
 & & \circ\ar@{-}[r] &  \circ\ar@{-}[dl]\ar@{-}[dr]\ar@{-}[r] & \circ \\ 
\circ\ar@{-}[r] & \circ\ar@{-}[r]\ar@{-}[ru] & \circ & \circ\ar@{-}[u]  
& \circ\ar@{-}[r] & \circ\ar@{-}[r]\ar@{-}[lu]  & \circ \\
 & & &  \circ\ar@{-}[ul]\ar@{-}[ur]\ar@{-}[d]  \\ 
 & & &  \circ 
}
$$ 
$$
\xymatrix@=2mm@R=2mm{ 
\circ\ar@{-}[d] \\ 
\circ\ar@{-}[d] \ar@{-}[r] & \circ\ar@{-}[d] \\ 
\circ\ar@{-}[d] \ar@{-}[r] & \circ\ar@{-}[d]\ar@{-}[r] & \circ\ar@{-}[d]   \\ 
\circ\ar@{-}[r]\ar@{-}[d] & \circ\ar@{-}[r]\ar@{-}[d] & \circ\ar@{-}[r]\ar@{-}[d] 
 &  \circ\ar@{-}[d]\ar@{-}[r] & \circ\ar@{-}[d]\ar@{-}[r] & \circ\\
\circ\ar@{-}[r] &  \circ\ar@{-}[r] & \circ\ar@{-}[r] &  \circ\ar@{-}[r]\ar@{-}[d] & \circ\\
 & & & \circ 
} \qquad 
\xymatrix@=2mm@R=2mm{ 
\circ\ar@{-}[d] \\ 
\circ\ar@{-}[d] \\ 
\circ\ar@{-}[d] \\ 
\circ\ar@{-}[d] \ar@{-}[r] & \circ\ar@{-}[d] \\ 
\circ\ar@{-}[d] \ar@{-}[r] & \circ\ar@{-}[d]\ar@{-}[r] & \circ\ar@{-}[d]   \\ 
\circ\ar@{-}[r]\ar@{-}[d] & \circ\ar@{-}[r]\ar@{-}[d] & \circ\ar@{-}[r]\ar@{-}[d] & 
  \circ\ar@{-}[r] & \circ\ar@{-}[r] & \circ\\
\circ\ar@{-}[r]\ar@{-}[d] &  \circ \ar@{-}[r]\ar@{-}[d] & \circ\\
 \circ\ar@{-}[d]\ar@{-}[r] & \circ  \\
\circ
}
$$ 
$$
\xymatrix@=2mm@R=0mm{ 
\circ\ar@{-}[dd] \\ 
& \\
\circ\ar@{-}[dd] \\ 
 & \\
\circ\ar@{-}[dd] && &&  \circ\ar@{-}[dd]\\
 & \\ 
\circ\ar@{-}[dd] && &&  \circ\ar@{-}[dd] \\
 & \\ 
\circ\ar@{-}[dd] \ar@{-}[rr] && \circ\ar@{-}[dd]\ar@{-}[rr] & & \circ\ar@{-}[dd]  \\
 & && \circ\ar@{-}[dd] \\ 
\circ\ar@{-}[dd] \ar@{-}[rr] & & \circ\ar@{-}[dd]\ar@{-}[rr]\ar@{-}[ru]
 & & \circ\ar@{-}[dd]   \\ 
 & && \circ\ar@{-}[dd] \\
 \circ\ar@{-}[rr]  && \circ\ar@{-}[rr]\ar@{-}[ru]  && \circ\ar@{-}[r]  &
 \circ\ar@{-}[r] & \circ\ar@{-}[r] & \circ\\ 
  &&& \circ\ar@{-}[dd] \\
 & \\
  &&& \circ 
} \qquad
\xymatrix@=2mm@R=2mm{ 
\circ\ar@{-}[d] \\ 
\circ\ar@{-}[d] \\ 
\circ\ar@{-}[d] \\ 
\circ\ar@{-}[r]\ar@{-}[d] & \circ\ar@{-}[r] & \circ\ar@{-}[r] 
 &  \circ \ar@{-}[r] & \circ \ar@{-}[r] & \circ \ar@{-}[r] & \circ \ar@{-}[r] & \circ\\
\circ }
$$ 
$$
\xymatrix@=2mm@R=0mm{ 
\circ\ar@{-}[dd] && &&  \circ\ar@{-}[dd] \\
 & \\ 
\circ\ar@{-}[dd] \ar@{-}[rr] && \circ\ar@{-}[dd]\ar@{-}[rr] & & \circ\ar@{-}[dd]  \\
 & && \circ\ar@{-}[dd] \\ 
\circ \ar@{-}[dd] \ar@{-}[rr] & & \circ\ar@{-}[dd]\ar@{-}[rr]\ar@{-}[ru]
 & & \circ\ar@{-}[dd]   \\ 
 & && \circ\ar@{-}[dd] \\
 \circ\ar@{-}[rr]  && \circ\ar@{-}[rr]\ar@{-}[ru]  && \circ \\ 
  &&& \circ\ar@{-}[dd] \\
 & \\
  &&& \circ\ar@{-}[dd] \\
 & \\
  &&& \circ 
} \qquad 
\xymatrix@=2mm@R=2mm{ 
 & & & \circ \ar@{-}[d]  \\ 
 \circ\ar@{-}[r] & \circ\ar@{-}[r] & \circ \ar@{-}[r] & \circ\ar@{-}[r] & \circ\ar@{-}[r] 
& \circ\ar@{-}[r] &  \circ
}
$$ 
$$
\xymatrix@=2mm@R=0mm{ 
  && \circ\ar@{-}[dd] \\
 &&  \\
  && \circ\ar@{-}[dd] \\
 && \\
 && \circ\ar@{-}[dd] \\
 & \circ\ar@{-}[dd]\ar@{-}[ru] \\
 & & \circ\ar@{-}[dd] \\
\circ \ar@{-}[dd] \ar@{-}[r]\ar@/_/@{-}[rrrr] & \circ\ar@{-}[dd]\ar@{-}[rr]\ar@{-}[ru]
 & &  \circ\ar@{-}[dd]\ar@{-}[r] & \circ\ar@{-}[r]\ar@{-}[dd] & \circ\ar@{-}[r]\ar@{-}[dd] 
 & \circ\ar@{-}[r]  & \circ\ar@{-}[r]& \circ\\ 
 & & \circ \\
 \circ \ar@{-}[r]\ar@{-}[dd]\ar@/_/@{-}[rrrr]  & \circ\ar@{-}[rr]\ar@{-}[ru]  && \circ\ar@{-}[r]  & 
    \circ\ar@{-}[dd]\ar@{-}[r] & \circ \\ 
 & \\
  \circ\ar@{-}[rrrr]\ar@{-}[dd]  & && &  \circ  \\
 & \\
   \circ 
} \qquad
\xymatrix@=2mm@R=2mm{ 
&&&&\circ\ar@{-}[d] \\ 
&\circ\ar@{-}[d] \ar@{-}[r] & \circ\ar@{-}[d] \ar@{-}[r] &  \circ\ar@{-}[d] \ar@{-}[r] & 
 \circ\ar@{-}[d] \ar@{-}[r] &  \circ\ar@{-}[d]   \\ 
&\circ\ar@{-}[d] \ar@{-}[r] & \circ\ar@{-}[d]\ar@{-}[r] & \circ\ar@{-}[d] \ar@{-}[r] & 
  \circ\ar@{-}[r] & \circ\ar@{-}[r] & \circ   \\ 
&\circ\ar@{-}[r]\ar@{-}[d] & \circ\ar@{-}[r]\ar@{-}[d] & \circ \\
\circ\ar@{-}[r] &\circ\ar@{-}[r]\ar@{-}[d] &  \circ\ar@{-}[d] \\
&\circ\ar@{-}[r] &  \circ\ar@{-}[d] \\
 & & \circ 
}
$$ 
$$
\xymatrix@=2mm@R=2mm{ 
&&&&\circ\ar@{-}[d] \\ 
& \circ\ar@{-}[d] \ar@{-}[r] & \circ\ar@{-}[d] \ar@{-}[r] &  \circ\ar@{-}[d] \ar@{-}[r] & 
 \circ\ar@{-}[d]    \\ 
&\circ\ar@{-}[d] \ar@{-}[r] & \circ \ar@{-}[d]\ar@{-}[r] & \circ\ar@{-}[d] \ar@{-}[r] & 
  \circ   \\ 
&\circ\ar@{-}[r]\ar@{-}[d] & \circ\ar@{-}[r]\ar@{-}[d] & \circ \\
\circ\ar@{-}[r] &\circ\ar@{-}[r] &  \circ 
} \qquad
\xymatrix@=2mm@R=2mm{ 
&&\circ\ar@{-}[d] \\ 
& \circ\ar@{-}[d] \ar@{-}[r] & \circ\ar@{-}[d] \ar@{-}[r] &  \circ\ar@{-}[d]  
  \\ 
\circ\ar@{-}[r] &\circ\ar@{-}[d] \ar@{-}[r] & \circ \ar@{-}[d]\ar@{-}[r] & \circ\ar@{-}[d] \ar@{-}[r] & 
  \circ   \\ 
&\circ\ar@{-}[r] & \circ\ar@{-}[r]\ar@{-}[d] & \circ \\
 & &  \circ 
}
$$ 
$$
\xymatrix@=2mm@R=2mm{ 
&&\circ\ar@{-}[d] \\ 
& \circ\ar@{-}[d] \ar@{-}[r] & \circ\ar@{-}[d] \ar@{-}[r] & 
 \circ\ar@{-}[d]\ar@{-}[r] &  \circ\ar@{-}[d] \ar@{-}[r] & \circ     \\ 
\circ\ar@{-}[r] & \circ  \ar@{-}[r] & \circ  \ar@{-}[r] & \circ\ar@{-}[d] \ar@{-}[r] & 
   \circ   \\ 
 & & &  \circ 
} \qquad
\xymatrix@=2mm@R=2mm{ 
\circ\ar@{-}[d] \\ 
\circ\ar@{-}[d]\ar@{-}[r] & \circ \\ 
\circ\ar@{-}[d] \\ 
\circ\ar@{-}[d] \ar@{-}[r] & \circ\ar@{-}[d] \\ 
\circ\ar@{-}[d] \ar@{-}[r] & \circ \ar@{-}[d]\ar@{-}[r] & \circ\ar@{-}[d]   \\ 
\circ\ar@{-}[r]\ar@{-}[d] & \circ\ar@{-}[r]\ar@{-}[d] & \circ\ar@{-}[r]\ar@{-}[d] &  \circ\\
\circ\ar@{-}[r]\ar@{-}[d] &  \circ \ar@{-}[r]\ar@{-}[d] & \circ\\
 \circ\ar@{-}[d]\ar@{-}[r] & \circ  \\
\circ
}
$$ 
$$
\xymatrix@=2mm@R=2mm{ 
 & & & \circ \ar@{-}[d]  \\ 
 \circ\ar@{-}[r] & \circ \ar@{-}[r]\ar@{-}[d] & \circ \ar@{-}[r] & \circ\ar@{-}[r] & \circ \\
 & \circ 
} \qquad
\xymatrix@=2mm@R=2mm{ 
&&&\circ\ar@{-}[d]\ar@{-}[r] & \circ\ar@{-}[d]\ar@{-}[r] &\circ\ar@{-}[d]\ar@{-}[r] 
 &\circ\ar@{-}[d]\ar@{-}[r] & \circ \\ 
&&& \circ\ar@{-}[d] \ar@{-}[d]\ar@{-}[r] &\circ\ar@{-}[d]\ar@{-}[r] & \circ\ar@{-}[r]\ar@{-}[d] & \circ \\ 
 &&&\circ\ar@{-}[d] \ar@{-}[r] & \circ\ar@{-}[d]\ar@{-}[r] & \circ  \\ 
\circ\ar@{-}[d] \ar@{-}[r] & \circ \ar@{-}[d]\ar@{-}[r] & \circ\ar@{-}[d]\ar@{-}[r] &
 \circ\ar@{-}[d]\ar@{-}[r] & \circ\ar@{-}[d]   \\ 
\circ\ar@{-}[r]\ar@{-}[d] & \circ\ar@{-}[r]\ar@{-}[d] & \circ\ar@{-}[r]\ar@{-}[d] &  \circ\ar@{-}[r] & \circ \\
\circ\ar@{-}[r]\ar@{-}[d] &  \circ \ar@{-}[r]\ar@{-}[d] & \circ & & 
 & \circ\ar@{-}@/_/[rruuuuu]\\
 \circ\ar@{-}[d]\ar@{-}[r] & \circ  \\
\circ\ar@{-}@/_/[rrrrruu]
}
$$ 
$$
\xymatrix@=2mm@R=0mm{ 
\circ\ar@{-}[dd]\ar@{-}@/_1pc/[dddddddd]  \ar@{-}[r] & \circ\ar@{-}[dd]\ar@{-}[rr] & & \circ\ar@{-}[dd]  \\
 &&  \\ 
\circ\ar@{-}[dd] \ar@{-}[r]  & \circ\ar@{-}[dd]\ar@{-}[rr]\ar@{-}[rd]
  && \circ\ar@{-}[dd]\ar@{-}[rd]     \\ 
  && \circ\ar@{-}[dd]\ar@{-}[rr] && \circ\ar@{-}[dd]\ar@{-}[r] & \circ \ar@{-}[dd]\\
 \circ\ar@{-}[r]  & \circ\ar@{-}[rr]\ar@{-}[rd]  & &
 \circ\ar@{-}[rd] \\ 
  && \circ\ar@{-}[dd]\ar@{-}[rr] && \circ\ar@{-}[dd]\ar@{-}[r] & \circ\ar@{-}[dd] \\
 & \\
  && \circ\ar@{-}[rr] && \circ\ar@{-}[r] & \circ\ar@{-}@/^1pc/[dddddddd] \\
\circ\ar@{-}[dd] \ar@{-}[r] & \circ\ar@{-}[dd]\ar@{-}[rr] & & \circ\ar@{-}[dd]  \\
 &&  \\ 
\circ\ar@{-}[dd] \ar@{-}[r]  & \circ\ar@{-}[dd]\ar@{-}[rr]\ar@{-}[rd]
  && \circ\ar@{-}[dd]\ar@{-}[rd]     \\ 
  && \circ\ar@{-}[dd]\ar@{-}[rr] && \circ\ar@{-}[dd]\ar@{-}[r] & \circ \ar@{-}[dd]\\
 \circ\ar@{-}[r]  & \circ\ar@{-}[rr]\ar@{-}[rd]  & &
 \circ\ar@{-}[rd] \\ 
  && \circ\ar@{-}[dd]\ar@{-}[rr] && \circ\ar@{-}[dd]\ar@{-}[r] & \circ\ar@{-}[dd] \\
 & \\
  && \circ\ar@{-}[rr] && \circ\ar@{-}[r] & \circ 
} \qquad \qquad 
\xymatrix@=2mm@R=2mm{ 
 & \\
&&&&&\circ\ar@{-}[d]\ar@{-}[r] & \circ\ar@{-}[d]\ar@{-}[r] & \circ\ar@{-}[d]  \\ 
&\circ\ar@{-}@/^2pc/[rrrrrru]&&&&\circ\ar@{-}[d]\ar@{-}[r] & \circ\ar@{-}[d]\ar@{-}[r] & \circ\ar@{-}[d] \\ 
&&\circ\ar@{-}[d] \ar@{-}[r] & \circ\ar@{-}[d] \ar@{-}[r] &  \circ\ar@{-}[d] \ar@{-}[r] & 
 \circ\ar@{-}[d] \ar@{-}[r] &  \circ\ar@{-}[d] & \circ   \\ 
&&\circ\ar@{-}[d] \ar@{-}[r] & \circ \ar@{-}[d]\ar@{-}[r] & \circ\ar@{-}[d] \ar@{-}[r] & 
  \circ\ar@{-}[r] &  \circ   \\ 
&&\circ\ar@{-}[r]\ar@{-}[d] & \circ\ar@{-}[r]\ar@{-}[d] & \circ \\
\circ\ar@{-}[r]\ar@{-}[d] &\circ\ar@{-}[r]\ar@{-}[d] &\circ\ar@{-}[r]\ar@{-}[d] & 
 \circ\ar@{-}[d] && \circ\ar@{-}@/_/[rruuu] \\
\circ\ar@{-}[r]\ar@{-}[d] &\circ\ar@{-}[r]\ar@{-}[d] &\circ\ar@{-}[r]&  \circ
  &&& \circ\ar@{-}@/_/[ruuuu]\ar@{-}@/_2pc/[uuuruuu] \\
\circ\ar@{-}[r]\ar@{-}@/_2pc/[rrrrrru]\ar@{-}@/^2pc/[ruuuuuu] & \circ\ar@{-}[r] & 
  \circ\ar@{-}@/_/[rrrru]\ar@{-}@/_/[rrruu]  \\
& 
}
$$ 
$$
\xymatrix@=2mm@R=2mm{ 
& & &&&&& \circ\ar@{-}[d] \\
&&&&&\circ\ar@{-}[d]\ar@{-}[r] & \circ\ar@{-}[d]\ar@{-}[r] & \circ\ar@{-}[d]\ar@{-}[r] & \circ \ar@{-}[d] \\ 
&&\circ\ar@{-}@/^/[rrrrruu] &&&\circ\ar@{-}[d]\ar@{-}[r] & \circ\ar@{-}[d]\ar@{-}[r] & \circ\ar@{-}[r] & \circ \\ 
&&&\circ\ar@{-}[d] \ar@{-}[r] & \circ\ar@{-}[d] \ar@{-}[r] &  \circ\ar@{-}[d] \ar@{-}[r] &   \circ\ar@{-}[d]   \\ 
&&&\circ\ar@{-}[d] \ar@{-}[r] & \circ \ar@{-}[d]\ar@{-}[r] & \circ \ar@{-}[r] &   \circ   \\ 
&\circ\ar@{-}[r]\ar@{-}[d] &\circ\ar@{-}[r]\ar@{-}[d] &\circ\ar@{-}[r]\ar@{-}[d] & 
 \circ\ar@{-}[d] & \\
&\circ\ar@{-}[r]\ar@{-}[d] &\circ\ar@{-}[r]\ar@{-}[d] &\circ\ar@{-}[r]&  \circ
  && \circ\ar@{-}@/_2pc/[rruuuuu]\ar@{-}@/_5pc/[ruuuuuu]   \\
\circ\ar@{-}[r]\ar@{-}@/_5pc/[rrrrrru]\ar@{-}@/^/[rruuuuu] &\circ\ar@{-}[r]\ar@{-}[d] & \circ\ar@{-}[d] \\
& \circ\ar@{-}[r]\ar@{-}@/_2pc/[rrrrruu]& \circ \\
& \\
&
} \qquad \qquad \qquad
\xymatrix@=2mm@R=2mm{ 
\circ\ar@{-}[d] \ar@{-}[r] & \circ \ar@{-}[d]\ar@{-}[r] & \circ\ar@{-}[d]\ar@{-}[r] &
 \circ\ar@{-}[d]\ar@{-}[r] & \circ\ar@{-}[r] & \circ\ar@{-}[r]\ar@{-}[d] 
  & \circ\ar@{-}[r]\ar@{-}[d] & \circ\ar@{-}[r]\ar@{-}[d] & \circ\ar@{-}[d]   \\ 
\circ\ar@{-}[r]\ar@{-}[d] & \circ\ar@{-}[r]\ar@{-}[d] & \circ\ar@{-}[r]\ar@{-}[d] &  \circ && \circ\ar@{-}[r] 
 & \circ\ar@{-}[r]\ar@{-}[d] & \circ\ar@{-}[r]\ar@{-}[d] & \circ\ar@{-}[d] \\
\circ\ar@{-}[r]\ar@{-}[d] &  \circ \ar@{-}[r]\ar@{-}[d] & \circ  
  && && \circ \ar@{-}[r] & \circ \ar@{-}[r] \ar@{-}[d] & \circ\ar@{-}[d] \\
 \circ \ar@{-}[r]\ar@{-}@/_1pc/[rrrr] & \circ && & \circ\ar@{-}@/_1pc/[rrrr] 
  & && \circ \ar@{-}[r] & \circ  \\
 & }
$$
$$
\xymatrix@=2mm@R=0mm{ 
& & & \circ\ar@{-}[dd]\ar@{-}[rr] && \circ\ar@{-}[r]\ar@{-}[dd] &\circ\ar@{-}[r]\ar@{-}[dd] &\circ  \\ 
 & \\
& & & \circ\ar@{-}[dd]\ar@{-}[rr] && \circ\ar@{-}[r]\ar@{-}[dd] &\circ\ar@{-}[dd]   \\
& \circ\ar@{-}[dd]  \ar@{-}[r]& \circ\ar@{-}[rr]\ar@{-}[ru]\ar@{-}[dd]
 & & \circ\ar@{-}[ru] \ar@{-}[dd]   \\ 
& & & \circ\ar@{-}[rr] && \circ\ar@{-}[r] & \circ \\
& \circ\ar@{-}[r]\ar@{-}[dd]  & \circ\ar@{-}[rr]\ar@{-}[ru]\ar@{-}[dd] 
 && \circ\ar@{-}[dd]\ar@{-}[ru] \\ 
 & \\
\circ\ar@{-}[r] &  \circ\ar@{-}[r]  & \circ\ar@{-}[rr] & & \circ
} \qquad 
\xymatrix@=2mm@R=0mm{ 
 & & \circ\ar@{-}[dd]\ar@{-}[rr] && \circ\ar@{-}[r]\ar@{-}[dd] &\circ\ar@{-}[r]\ar@{-}[dd] &\circ\ar@{-}[dd]  \\ 
 & \\
 & & \circ\ar@{-}[dd]\ar@{-}[rr] && \circ\ar@{-}[r]\ar@{-}[dd] &\circ\ar@{-}[dd] &\circ\ar@{-}[dd]  \\
 \circ\ar@{-}[dd]  \ar@{-}[r]& \circ\ar@{-}[rr]\ar@{-}[ru]\ar@{-}[dd]
 & & \circ\ar@{-}[ru] \ar@{-}[dd]   \\ 
 & & \circ\ar@{-}[rr] && \circ\ar@{-}[r] & \circ &\circ\ar@{-}[dd] \\
 \circ \ar@{-}[r]\ar@{-}[dd]  & \circ\ar@{-}[rr]\ar@{-}[ru]\ar@{-}[dd] 
 && \circ\ar@{-}[dd]\ar@{-}[ru] \\ 
 & &&&&&\circ\\
  \circ\ar@{-}[r]\ar@{-}[dd]\  & \circ\ar@{-}[rr] & & \circ\\ 
  & \\
 \circ\ar@{-}[r] & \circ\ar@{-}[rr] && \circ\ar@{-}[r] & \circ\ar@{-}[rruuu]  
}
$$ 
\end{itemize}

%%%%%%%%%%%%%%%
%
\section{Component sizes}

The following tables show the number of connected components for the flip graph, for the square, pentagon, hexagon, heptagon, octagon, and nonagon. They were found using a computer search, after generating all triangulations using the same recursive method as Figure~\ref{fig:triangulationscountingoctagon}, and then testing which pairs differ by a flip.

\noindent
Square: $n = 2$\\
\begin{tabular}{|l|r|r|}
%\hline 
%$n=2$ &   &  \\
 \hline 
size  & 1 & 2  \\
number & 4 & 2  \\
\hline 
\end{tabular}\\

\noindent
Pentagon: $n = 3$\\
\noindent
\begin{tabular}{|l|r|r|}
%\hline 
%$n=3$ &   &  \\
 \hline 
size  & 1 & 3  \\
number & 10 & 10  \\
\hline 
\end{tabular}

\vskip.2cm

\noindent
Hexagon: $n = 4$\\
\noindent
\begin{tabular}{|l|r|r|r|r|}
 \hline 
size  & 1 & 4 & 5 & 6  \\
number & 28 & 16 & 12 & 12\\
\hline 
\end{tabular}

\vskip.2cm

\noindent 
Heptagon: $n = 5$\\
\noindent
\begin{tabular}{|l|r|r|r|r|r|r|}
 \hline 
size  & 1 & 5 & 6 & 9 & 10 & 12  \\
number & 84 & 14 & 28 & 42 & 14 & 42\\
\hline 
\end{tabular}

\vskip.2cm

\noindent
Octagon: $n = 6$ \\
\begin{tabular}{|l|r|r|r|r|r|r|r|r|r|r|r|r|r|r|r|r|r|r|r|r|r|r|}
 \hline 
size  & 1 & 6 & 7 & 8 & 10 & 12 & 13 & 14 & 15 & 16 & 18 & 19 & 20 & 21 & 22 & 23 & 26 & 28 & 29 & 32 & 34 & 36  \\
number & 264 & 16 & 16 & 16 & 16 & 64 & 8 & 8 & 16 & 32 & 32 & 64 & 40 & 16 & 32 & 32 & 16 & 8 & 16 & 2 & 8 & 4\\
\hline 
\end{tabular}

\vskip.2cm

\noindent
Nonagon \\
\begin{tabular}{|l|r|r|r|r|r|r|r|r|r|r|r|r|r|r|r|r|r|r|r|r|r|r|r|r|r|r|r|r|r|r|r|r|r|r|}
 \hline 
size  & 1 & 7 & 9 & 13 & 15 & 17 & 18 & 21 & 23 & 27 & 28 & 29 & 31 & 32 & 33 & 34 & 35 %& 36 & 37 & 38 & 41 & 42 & 44 & 45 & 46 & 53 & 55 & 57 & 59 & 61 & 66 & 70 & 71 & 79 
\\
number & 858 & 18 & 36 & 36 & 54 & 36 & 36 & 18 & 72 & 126 & 72 & 6 & 54 & 36 & 18 & 72 & 18 %& 108 & 36 & 72 & 36 & 36 & 36 & 108 & 36 & 54 & 36 & 18 & 54 & 36 & 36 & 36 & 18 & 6
\\
\hline 
\end{tabular}

\vskip.2cm

\noindent
Nonagon, continued\\
\begin{tabular}{|l|r|r|r|r|r|r|r|r|r|r|r|r|r|r|r|r|r|r|}
\hline
size &  36 & 37 & 38 & 41 & 42 & 44 & 45 & 46 & 53 & 55 & 57 & 59 & 61 & 66 & 70 & 71 & 79  \\
number  & 108 & 36 & 72 & 36 & 36 & 36 & 108 & 36 & 54 & 36 & 18 & 54 & 36 & 36 & 36 & 18 & 6\\
\hline
\end{tabular}

\subsection*{Acknowledgment} 

This write-up is a result of the ``Count Me In''  summer school, organised by David Jordan, Milena Hering, and Nick Sheridan, funded by ICMS, University of Edinburgh and Glasgow Mathematical Journal Trust.

Diana Bergerova, Jenni Voon and Lejie Xu thank Karin Baur for suggesting the topic of the paper and supervising the work as well as for providing diagrams of connected components from work with Mark Parsons. They are also grateful to their tutor and mentor Stefania Lisai for helping the project with her comments and notes. 

Karin Baur is supported by a Royal Society Wolfson Award, RSWF$\backslash$R1$\backslash$180004 and by the EPSRC Programme Grant EP/W007509/1.

%%%%%%%%%%%% bibliography %%%%%%%%%%%%
\medskip

\bibliographystyle{alphaurl}
\bibliography{references}

\begin{thebibliography}{AHK77}

\bibitem[AH77]{10.1215/ijm/1256049011}
K.~Appel and W.~Haken.
\newblock {Every planar map is four colorable. Part I: Discharging}.
\newblock {\em Illinois Journal of Mathematics}, 21(3):429 -- 490, 1977.
\newblock \href {https://doi.org/10.1215/ijm/1256049011}
  {\path{doi:10.1215/ijm/1256049011}}.

\bibitem[AHK77]{10.1215/ijm/1256049012}
K.~Appel, W.~Haken, and J.~Koch.
\newblock {Every planar map is four colorable. Part II: Reducibility}.
\newblock {\em Illinois Journal of Mathematics}, 21(3):491 -- 567, 1977.
\newblock \href {https://doi.org/10.1215/ijm/1256049012}
  {\path{doi:10.1215/ijm/1256049012}}.

\bibitem[GP02]{GRAVIER2002817}
Sylvain Gravier and Charles Payan.
\newblock Flips signés et triangulations d’un polygone.
\newblock {\em European Journal of Combinatorics}, 23(7):817--821, 2002.
\newblock URL:
  \url{https://www.sciencedirect.com/science/article/pii/S0195669802906013},
  \href {https://doi.org/https://doi.org/10.1006/eujc.2002.0601}
  {\path{doi:https://doi.org/10.1006/eujc.2002.0601}}.

\bibitem[Hat91]{hatcher}
Allen Hatcher.
\newblock On triangulations of surfaces.
\newblock {\em Topology and its Applications}, 40(2):189--194, 1991.
\newblock URL:
  \url{https://www.sciencedirect.com/science/article/pii/016686419190050V},
  \href {https://doi.org/https://doi.org/10.1016/0166-8641(91)90050-V}
  {\path{doi:https://doi.org/10.1016/0166-8641(91)90050-V}}.

\end{thebibliography}

\end{document}